\documentclass[11pt,letterpaper]{article}
\usepackage{tikz}
\usetikzlibrary{shapes.geometric}
\usepackage{amsfonts, amsmath, amssymb, amscd, amsthm, color, graphicx, mathabx, mathrsfs, wasysym, setspace, mdwlist, calc, float, mathtools, tikz-cd}

\usepackage{tocloft}
\usepackage{enumerate}

\usepackage{xcolor}

\usepackage{hyperref,cancel}

\hypersetup{colorlinks,
    linkcolor={red!50!black},
    citecolor={blue!80!black},
    urlcolor={blue!80!black}}
\usepackage{float}

\numberwithin{equation}{section}

\newcommand{\La}{\Lambda}

\newcommand{\Ga}{\Gamma}

\renewcommand{\ll }{\langle\hspace{-.7mm}\langle }
\newcommand{\rr }{\rangle\hspace{-.7mm}\rangle }

\newcommand{\ra}{\rightarrow}
\newcommand{\ca}{\curvearrowright}

\newcommand{\A}{\mathcal A}
\newcommand{\B}{\mathcal B}

\newcommand{\D}{\mathcal D}

\renewcommand{\L}{\mathcal L}
\newcommand{\M}{\mathcal M}
\newcommand{\N}{\mathcal N}
\renewcommand{\P}{\mathcal P}
\newcommand{\Q}{\mathcal Q}
\newcommand{\R}{\mathcal R}
\renewcommand{\S}{\mathcal S}
\newcommand{\T}{\mathcal T}

\newcommand{\X}{\mathcal X}
\newcommand{\Y}{\mathcal Y}

\newcommand{\W}{\mathcal W}

\newcommand{\sC}{\mathscr C}
\newcommand{\sD}{\mathscr D}
\newcommand{\sE}{\mathscr E}
\newcommand{\sF}{\mathscr F}
\newcommand{\sG}{\mathscr G}
\newcommand{\sH}{\mathscr H}

\newcommand{\sN}{\mathscr N}
\newcommand{\sP}{\mathscr P}

\newcommand{\sS}{\mathscr S}
\newcommand{\sT}{\mathscr T}
\newcommand{\sU}{\mathscr U}
\newcommand{\sV}{\mathscr V}

\newcommand{\sZ}{\mathscr Z}
\newcommand{\sW}{\mathscr W}

\theoremstyle{plain}
\newtheorem{main}{Theorem}

\newtheorem{mcor}[main]{Corollary}

\newtheorem{thm}{Theorem}[section]
\newtheorem*{thm*}{Theorem}
\newtheorem{cor}[thm]{Corollary}
\newtheorem{lem}[thm]{Lemma}

\newtheorem{prop}[thm]{Proposition}

\theoremstyle{definition}
\newtheorem{defn}[thm]{Definition}

\theoremstyle{plain}
\newtheorem{rem}[thm]{Remark}

\begin{document}

\title{Rigidity for von Neumann algebras of  graph product groups. I. Structure of automorphisms}
\author{Ionu\c t Chifan, Michael Davis, Daniel Drimbe }
\date{}

\maketitle
\begin{abstract} \noindent In this paper we study various rigidity aspects of the von Neumann algebra $L(\Ga)$ where  $\Ga$ is a graph product group whose underlying graph is a certain cycle of cliques and the vertex groups are  wreath-like product property (T) groups. Using an approach that combines methods from Popa's deformation/rigidity theory with new techniques pertaining to graph product algebras, we describe all symmetries of these von Neumann algebras and reduced C$^*$-algebras by establishing formulas in the spirit of Genevois and Martin's results on automorphisms of graph product groups. 
 \end{abstract}
\section{Introduction}

Graph product groups were introduced by E. Green  \cite{Gr90} in her PhD thesis as natural generalizations of classical right-angled Artin and Coxeter groups. Their study has become a trendy subject over the years as they play key roles in various branches of topology and group theory. For example, over the last decade graph products groups have been intensively studied through the lens of geometric group theory resulting in many new important
discoveries---\cite{Ag13, HW08, W11, MO13, AM10}, just to enumerate a few. 

In a different direction, by using techniques from measured group theory
 interesting orbit equivalence rigidity results have been obtained for measure preserving actions on probability spaces of specific classes of graph product groups, including many right-angled Artin groups  \cite{HH20,HH21}.

General graph product groups were
considered in the analytic framework of von Neumann algebras for the first time in \cite{CF14}. Since then several structural results such as strong solidity, absence/uniqueness of Cartan subalgebras, and classification of their tensor decompositions have been established in
 \cite{CF14, Ca16, CdSS17, DK-E21,CK-E21} for von Neumann algebras arising from these groups and their actions on probability spaces. Since general graph product groups display such a rich combinatorial structure, much remains to be done in this area, and understanding how this complexity is reflected in the von Neumann algebras remains mysterious.

This paper is the first of two which  will investigate new rigidity aspects for von Neumann algebras of graph product groups through the powerful Popa's deformation/rigidity theory \cite{Po06}. This theory provides  a novel conceptual framework through which a large number of impressive structural and rigidity results for von Neumann algebras have been discovered over the last two decades; see the surveys \cite{Po06,Va10a,Io12b, Io17}.  These papers will analyze new inputs in this theory from the perspective of graph products algebras. In the first paper we completely  describe the structure of all $\ast$-isomorphisms between von Neumann algebras arising from a large class of graph product groups (see Section \ref{CC1}). {In the second one \cite{CDD23}, we investigate superrigidity aspects of these von Neumann alegbras.}

\subsection{Statements of the main results}

To properly introduce our results we briefly recall the construction of graph product groups. Let  $\mathscr G=(\mathscr V,\mathscr E)$ be a finite {\it simple graph} (i.e. $\sG$ does not admit more than one edge between any two vertices and no edge starts and ends at the same vertex). The \emph{graph product group} $\Ga={\mathscr G}\{\Ga_v\}$ of a given  family  of \emph{vertex groups} $\{\Ga_v\}_{v\in \mathscr V}$ is the quotient of the free product $\ast_{v\in \mathscr V} \Ga_v$ by the relations $[\Ga_u,\Ga_v]=1$ whenever $u$ and $v$  are connected by an edge, $(u,v)\in \mathscr E$. Thus, graph products can be thought of as groups that ``interpolate'' between the direct product $\times_{v\in \mathscr V} \Ga_v$ (when $\mathscr G$ is complete) and the free product $\ast_{v\in \mathscr V} \Ga_v$ (when $\mathscr G$ has no edges).

For any subgraph $\sH= (\sU,\sF)$ of $\sG$ we  denote by $\Gamma_\sH$ the subgroup generated by $\Gamma_\sH =\langle \Gamma_u \,:\,u\in \sU\rangle $ and we call it the \emph{full subgroup} of $\sG\{\Gamma_v\}$ corresponding to $\sH$. A \emph{clique} $\sC$ of $\sG$ is a maximal, complete subgraph of $\sG$. The set of cliques of $\sG$ will be denoted by ${\rm cliq}(\sG)$. The full subgroups $\Ga_\sC$ for $\sC\in \rm cliq(\sG)$ are called the clique subgroups of $\sG\{\Ga_v\}$.  

In this paper we are interested in graph product groups arising from a specific class of graphs which we introduce next. 
A graph  $\sG$ is called a \emph{simple cycle of cliques (abbrev.\ ${\rm CC}_1$)} if there is an enumeration of its clique set $ {\rm cliq}(\sG)= \{\sC_1, ..., \sC_n\} $ with $n\geq 4$ such that the subgraphs $\sC_{i,j} :=\sC_i \cap \sC_j$ satisfy: 
\begin{equation*} \begin{split}
    &\sC_{i,j}=\begin{cases} \emptyset,  \text{ if }  \hat i-\hat j\in \mathbb Z_{n} \setminus \{ \hat 1, \widehat {n-1}\}\\
\neq \emptyset, \text{ if }  \hat i-\hat j\in \{\hat 1, \widehat {n-1}\} 
\end{cases} \text{ and }\\
&\mathscr{C}^{\rm int}_i:=\sC_i \setminus  \left(\sC_{i-1,i} \cup \sC_{i,i+1}\right) \neq \emptyset\text{, for all }i\in\overline{1,n}, \text{ with conventions } 0=n \text{ and } n+1=1.\end{split}
  \end{equation*}
Note this automatically implies the cardinality $|\sC_i |\geq 3$ for all $i$.  Also such an enumeration $ {\rm cliq}(\sG)= \{\sC_1,..., \sC_n\} $ is called \emph{a consecutive cliques enumeration}. A basic example of such a graph is any simple, length $n$, cycle of triangles $\sF_n=(\sV_n,\sE_n)$, which essentially looks like a flower shaped graph with $n$ petals:    

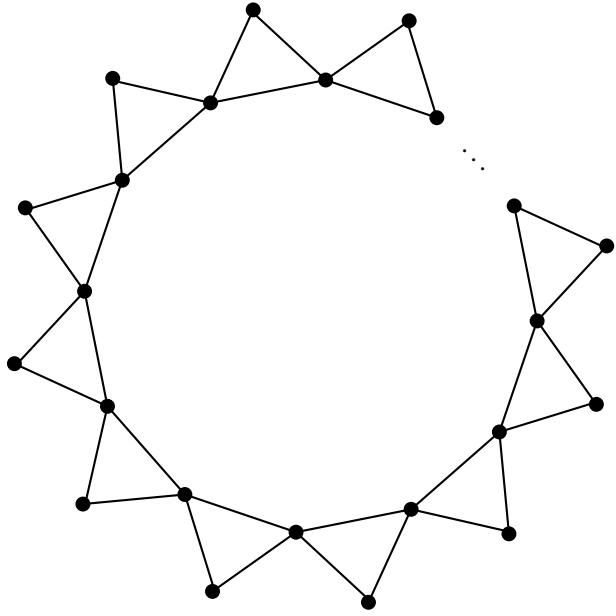
\begin{equation}\label{flower}
\begin{tikzpicture}[
regular polygon colors/.style 2 args={
    append after command={%
        \pgfextra
        \foreach \i [count=\ni, remember=\ni as \lasti (initially #1)] in {#2}{
            \draw[thick,\i] (\tikzlastnode.corner \lasti) --(\tikzlastnode.corner \ni);}
        \endpgfextra
    }
},
]

\node[thick, minimum size={2*3cm},regular polygon,regular polygon
 sides=12, rotate=11.25, regular polygon colors={12}{black, black, black, black, black, black, black, black, black, black, black, white}] at
 (3,3) (16-gon)[text=blue] {};

\node[thick, draw=white,minimum size={2*4cm},regular polygon,regular polygon
 sides=12, rotate=-4] at
 (3,3) (sixteengon)[text=blue] {};

\node[thick, draw=white,minimum size={2*4cm},regular polygon,regular polygon
 sides=12, rotate=26.5] at
 (3,3) (6teengon)[text=blue] {};

\path (16-gon.corner 12) -- node[auto=false, rotate=-45]{\ldots} (16-gon.corner 11);

\draw[thick] (6teengon.corner 12) -- (16-gon.corner 12);
 \node [fill=black,circle,inner sep=2pt] at (6teengon.corner 10){};
  
 \foreach \p [count=\n] in {1,...,12}{
 \node [fill=black,circle,inner sep=2pt] at (16-gon.corner \n){};
}
 \foreach \p [count=\n] in {3,...,12}{
 \node [fill=black,circle,inner sep=2pt] at (sixteengon.corner \n){};
\draw[thick] (6teengon.corner \n) -- (16-gon.corner \n);
}
 \foreach \p [count=\n] in {1,...,9}{
 \node [fill=black,circle,inner sep=0pt] at (6teengon.corner \n){};
}
\foreach \p [count=\n] in {1,...,11}{
 \draw[thick] (sixteengon.corner \n) -- (16-gon.corner \n);
}

\end{tikzpicture}
\end{equation}

\noindent In fact any graph in ${\rm CC}_1$ is a two-level clustered graph that is a specific retraction of $\mathscr F_n$; for more details the reader may consult Section \ref{CC1}.


The goal of this paper is to describe the structure of all $\ast$-isomorphisms between \emph{graph product group von Neumann algebras} (i.e. group von Neumann algebras arising from graph product groups) where the underlying graphs  belong to ${\rm CC}_1$.  To introduce our results, we first highlight a canonical family of {$\ast$-isomorphisms} between these algebras that are analogous to the graph product groups situation. Let $\sG,\sH \in {\rm CC}_1$ be isomorphic graphs and fix $\sigma : \sG \ra \sH$ an isometry. Let ${\rm cliq}( \sG) =\{ \sC_1,  \ldots ,\sC_n\}$ be a consecutive cliques  enumeration. Let $\Gamma_\sG$ and $\Lambda_\sH$ be graph product groups and assume that for every $i\in \overline{1,n} $ there are $\ast$-isomorphisms $\theta_{i-1,i}: \L(\Gamma_{\sC_{i-1,i}})\ra \L(\Lambda_{\sC_{\sigma (\sC_{i-1,i})}}) $,   $\xi_{i}:\L(\Gamma_{\sC^{\rm int}_i})\ra \L(\Lambda_{\sigma(\sC^{\rm int}_i)})$ and $\theta_{i,i+1}:\L(\Gamma_{\sC_{i,i+1}})\ra \L(\Lambda_{\sC_{\sigma (\sC_{i,i+1})}}) $; here and in what follows we convene as before that $n=0$ and $n+1=1$.  Results in Section \ref{locisomsec} show these $\ast$-isomorphisms induce a unique $\ast$-isomorphism $\phi_{\theta,\xi, \sigma}:\L(\Gamma_\sG)\ra  \L(\Lambda_\sH)$ defined as 

\begin{equation}\label{branchaut'} \phi_{\theta,\xi, \sigma}(x)=\begin{cases}
\theta_{i-1,i}(x), \text{ if } x\in \L(\Gamma_{\sC_{i-1,i}})\\
\xi_i(x), \qquad \text{if }  x\in \L(\Gamma_{\sC^{\rm int}_i })
\end{cases} 
\end{equation}
for all $i\in\overline{1,n}$. 

When $\Gamma_\sG =\La_\sH$ this construction yields a group of  $\ast$-automorphisms of $\L(\Gamma_\sG)$, denoted by ${\rm Loc}_{\rm c,g}(\L(\Gamma_\sG))$. We also denote by ${\rm Loc}_{\rm c}(\L(\Gamma_\sG))$ the subgroup of all local automorphisms satisfying $\sigma={\rm Id}$. Notice that ${\rm Loc}_{\rm c}(\L(\Gamma_\sG))\cong \oplus_i {\rm Aut}(\L(\Gamma_{\sC_{i-1,i}}))\oplus {\rm Aut}(\L(\Gamma_{\sC^{\rm int}_{i}}))$ also ${\rm Loc}_{\rm c}(\L(\Gamma_\sG))\leqslant {\rm Loc}_{\rm c,g}(\L(\Gamma_\sG))$ has finite index. 

Next, we highlight a class of automorphisms in ${\rm Loc}_{\rm c}(\L(\Gamma_\sG))$ needed to state our main results. Consider $n$-tuples $a =(a_{i,i+1})_i$ and $b= (b_i)_i$ of nontrivial unitaries  $a_{i,i+1} \in \L(\Gamma_{\sC_{i-1,i}})$ and $b_i\in \L(\Gamma_{\sC^{\rm int}_i})$, for every $i\in\overline{1,n}$. If in  \eqref{branchaut'} we let $\theta_{i,i+1}= {\rm ad} (a_{i,i+1}) $ and $\xi_i ={\rm ad} (b_i)$ then the corresponding local automorphism $\phi_{\theta,\xi, {\rm Id}}$ is most of the time an outer automorphism of $\L(\Gamma)$ and will be denoted by $\phi_{a,b}$ throughout the paper. These automorphisms form a normal subgroup denoted by ${\rm Loc}_{\rm c, i}(\L(\Gamma_\sG))\lhd{\rm Loc}_{\rm c}(\L(\Gamma_\sG))$ (see Section \ref{locisomsec} for more details).

Developing an approach which combines outgrowths of prior methods in Popa's deformation/rigidity theory \cite{IPP05} with a new technique on analyzing cancellation in cyclic relations of graph von Neumann algebras (Section \ref{vNcyclicrel}) we are able to describe of all $\ast$-isomorphisms between these algebras solely in terms of the aforementioned local isomorphisms. This can be viewed as a von Neumann algebra counterpart of very general and deep results of Genevois-Martin \cite[Corollary C]{GM19} from geometric group theory describing the structure of the automorphisms of graph products groups.

\begin{main}\label{A}\label{symmetries2'} Let $\sG,\sH\in {\rm CC}_1$ and $\Gamma= \sG\{\Gamma_v\}$, $\Lambda=\sH\{\Lambda_w\}$ be graph products such that: 
\begin{enumerate} \item $\Gamma_v$ and $\Lambda_w$ are icc property (T) groups for all $v\in \sV$, $w \in \sW$;
\item There is a class $\mathcal C$ of countable groups which satisfies the s-unique prime factorization property (see Definition \ref{supf}) for which $\Gamma_v$ and $\Lambda_w$ belong to $\mathcal C$, for all $v\in \sV$, $w\in \sW$.
 \end{enumerate} 

\noindent 
Let $t>0$ and let $\Theta : \L(\Gamma)^t \ra  \L(\Lambda)$ be any $\ast$-isomorphism. Then $t=1$ and one can find an isometry $\sigma: \sG \ra \sH$, $\ast$-isomorphisms $\theta_{i-1,i}:\L(\Gamma_{\sC_{i-1,i}})\ra \L(\Gamma_{\sC_{\sigma (\sC_{i-1,i})}}) $,  $\xi_{i}:\L(\Gamma_{\sC^{\rm int}_i})\ra \L(\Gamma_{\sigma(\sC^{\rm int}_i)})$ for all $i\in\overline{1,n}$,  and a unitary $u\in \L (\Lambda)$ such that $\Theta = {\rm ad}(u)\circ \phi_{\theta,\xi, \sigma} $. 
\end{main}

This theorem applies to fairly large classes of property (T) vertex groups, including: all fibered Rips constructions considered in \cite{CDK19,CDHK20} and all wreath-like product groups $\mathcal W \mathcal R (A,B\ca I)$ where $A$ is either abelian or icc, $B$ is an icc subgroup of a hyperbolic group, and the action $B \ca I$ has amenable stabilizers, \cite{CIOS21}. The result also implies  that the fundamental group \cite{MvN36} of these graph product group II$_1$ factors is always trivial; this means that if $\Gamma$ is a graph product group as in Theorem \ref{A}, then $\{t>0|\;\L(\Gamma)^t\cong \L(\Gamma) \}=1$.
Recall that Popa used his deformation/rigidity theory for obtaining  the first examples of II$_1$ factors with trivial fundamental group \cite{Po01}, hence answering a longstanding open problem of Kadison, see \cite{Ge03}. Subsequently, a large number of striking  results on computations of fundamental groups of II$_1$ factors were obtained,
see  the introduction of \cite{CDHK20}.
To our knowledge, Theorem \ref{A} provides the first instance of computing the fundamental group for non-trivial graph product von Neumann algebras which is not a tensor product. 

Specializing Theorem \ref{symmetries2'} to the case when the vertex groups $\Ga_v$ and $\La_w$ are the property (T) wreath-like product groups as in \cite[Theorem 8.4]{CIOS21} we obtain a fairly concrete description of all such isomorphisms between these graph product group von Neumann algebras; namely, they appear as compositions between the canonical group-like isomorphisms  and the  clique-inner local automorphisms  of $\L(\La)$ described above.

\begin{main}\label{B}\label{symmetries3'} Let $\sG, \sH\in {\rm CC}_1$ and let $\Ga= \sG\{\Gamma_v\}$, $\Lambda=\sH\{\Lambda_w\}$ be graph product groups where all vertex groups $\Gamma_v,\Lambda_w$ are property (T) wreath-like product groups of the form $\W\R (A, B \ca I)$ where $A$ is abelian, $B$ is an icc  subgroup of a hyperbolic group, and $B \curvearrowright I$ has infinite orbits. 

\noindent Then for any $t>0$ and  $\ast$-isomorphism  $\Theta:\L(\Gamma)^t\ra \L(\Lambda)$ we have $t=1$ and one can find a character $\eta \in {\rm Char}( \Gamma)$, a group isomorphism  $\delta \in {\rm Isom}(\Gamma,\Lambda)$, a $\ast$-automorphism  $\phi_{a,b}\in {\rm Loc}_{\rm c, i}(\L(\La))$ and a unitary $u\in \L(\Lambda)$ such that $\Theta= {\rm ad}(u)\circ \phi_{a,b}\circ \Psi_{\eta, \delta }$.    
\end{main}

In the statement of Theorem \ref{B} and also throughout the paper, given a character $\eta \in {\rm Char}( \Gamma)$ and a group isomorphism $\delta \in {\rm Isom}(\Gamma,\Lambda)$, we denote by $\Psi_{\eta,\delta}$ the  $\ast$-isomorphism from $\L(\Gamma)$ to $\L(\Lambda)$ given by $\Psi_{\eta,\delta}(u_g)=\eta(g) v_{\delta(g)}$, for any $g\in \Gamma$. Here, $\{u_g \,:\,g\in \Gamma\}$ and $\{v_h\,:\,h\in \Lambda\}$ are the canonical group unitaries of $\L(\Gamma)$ and $\L(\Lambda)$, respectively.


To this end we recall that in \cite{CIOS21} it was shown that the property (T) regular wreath-like products covered by the previous theorem can be chosen to have trivial abelianization and prescribed finitely presented outer automorphism groups. Using this, Theorem \ref{symmetries3'} yields the following: 

\begin{mcor}\label{C} Let $\sG\in{\rm CC}_1$ and fix ${\rm cliq}(\sG)=\{\sC_1,\ldots, \sC_n\}$ a consecutive enumeration of its cliques. Let $\Gamma= \sG\{\Gamma_v\}$ be any graph product groups as in Theorem \ref{B}. Assume in addition that its vertex groups are pairwise non-isomorphic, have trivial abelianization and trivial outer automorphisms.  Then the outer automorphisms satisfy the following formula:  $${\rm Out} (\L(\Gamma))\cong  \oplus^n_{i=1} \sU(\L(\Gamma_{\sC_{i-1,i}})) \oplus \sU(\L(\Gamma_{\sC^{\rm int}_{i}})).$$

\end{mcor}

By applying Corollary \ref{C} to the case when the underlying graph $\sG$ is the $n$-petals flower shaped $\mathscr F_n=(\sV_n, \sE_n)$ (fig.\ \eqref{flower}), we obtain the slimmest types of outer automorphisms groups one could have in this setup. Namely, we deduce that ${\rm Out} (\L(\Gamma))\cong  \oplus_{v\in \sV_n} \sU(\L(\Gamma_{v})).$


We conclude our introduction by describing in Corollary \ref{D} all $*$-isomorphisms of the reduced C$^*$-algebras of graph product groups  that we considered in Theorem \ref{B}. This result can be seen as a C$^*$-algebraic version of Genevois and Martin's result \cite[Corollary C]{GM19}.

\begin{mcor}\label{D}
    Let $\sG, \sH\in {\rm CC}_1$ and let $\Ga= \sG\{\Gamma_v\}$, $\Lambda=\sH\{\Lambda_w\}$ be graph product groups as in Theorem \ref{B}. 

\noindent Then for any  $\ast$-isomorphism  $\Theta:C_r^*(\Gamma)\ra C_r^*(\Lambda)$ there exist a character $\eta \in {\rm Char}( \Gamma)$, a group isomorphism  $\delta \in {\rm Isom}(\Gamma,\Lambda)$, a  $\ast$-automorphism  $\phi_{a,b}\in {\rm Loc}_{\rm c, i}(\L(\La))$ and a unitary $u\in \L(\Lambda)$ such that $\Theta= {\rm ad}(u)\circ \phi_{a,b}\circ \Psi_{\eta, \delta }$.
\end{mcor}

In fact, this result is a consequence of Theorem \ref{B} since the graph product groups that we consider have trivial amenable radical (see Lemma \ref{trivial.amenable.radical}), and consequently, their reduced C$^*$-algebras have unique trace \cite{BKKO14}.

{\bf Acknowledgements.} We would like to thank Stefaan Vaes for helpful comments. We would also like to thank the referee for all the comments and suggestions that greatly improved the exposition of the paper. I.C. was partially supported by NSF FRG Grant \#1854194 and NSF Grant \#DMS-2154637; D.D. was supported by the postdoctoral fellowship fundamental research 12T5221N of the Research Foundation Flanders.

\section{Preliminaries}

\subsection{Terminology}
Throughout this document all von Neumann algebras are denoted by calligraphic letters e.g. $\M$, $\N$, $\P$, $\Q$, etc. 
All von Neumann algebras $\M$ considered in this document will be tracial, i.e. endowed with a unital, faithful, normal linear functional $\tau:\M\rightarrow \mathbb C$  satisfying $\tau(xy)=\tau(yx)$ for all $x,y\in \M$. This induces a norm on $\M$ by the formula $\|x\|_2=\tau(x^*x)^{1/2}$ for all $x\in \M$. The $\|\cdot\|_2$-completion of $\M$ will be denoted by $L^2(\M)$. 

Given a von Neumann algebra $\M$, we will denote by $\mathscr U(\M)$ its unitary group and by $\mathcal Z(\M)$ its center.  Given a unital inclusion $\N\subset \M$ of von Neumann algebras we denote by $\N'\cap \M =\{ x\in \M \,:\, [x, \N]=0\}$ the relative commmutant of $\N$ inside $\M$ and by $\mathscr N_\M(\N)=\{ u\in \mathscr U(\M)\,:\, u\N u^*=\N\}$ the normalizer of $\N$ inside $\M$. We say that the inclusion $\N$ is regular in $\M$ if $\mathscr N_{\M}(\N)''=\M$ and
irreducible if $\N'\cap \M=\mathbb C 1$.

\subsection{Graph product groups}
In this preliminary section we briefly recall the notion of graph product groups introduced by E. Green  \cite{Gr90} while also highlighting some of its features that are relevant to this article. Let $\sG=(\sV,\sE)$ be a finite {simple graph}, where $\sV$ and $\sE$  denote its vertex and edge sets, respectively. Let $\{\Gamma_v\}_{v\in\sV}$ be a family of groups called vertex groups. The graph product group associated with this data, denoted by $\sG\{\Gamma_v,v \in \sV\}$ or simply $\sG\{\Gamma_v\}$, is the group generated  by $\Gamma_v$, $v\in \sV$  with the only relations being $[\Gamma_u, \Gamma_v] = 1$, whenever $(u,v)\in \sE$. 
Given any subset $\sU\subset \sV$, the subgroup $\Gamma_\sU =\langle \Gamma_u \,:\,u\in \sU\rangle $ of $\sG\{\Gamma_v,v\in \sV\}$ is called a {\it full subgroup}. This can be identified  with the graph product $\sG_\sU\{\Gamma_u,u \in \sU\}$ corresponding to the subgraph $\sG_\sU$ of $\sG$, spanned by the vertices of $\sU$. 
For every $v \in \sV$ we denote by ${\rm lk}(v)$ the subset of vertices $w\neq v$ so that $(w,v)\in \sE$. Similarly, for every $\sU 
\subseteq \sV$ we denote by ${\rm lk}(\sU) = \cap _{u\in \sU}{\rm lk}(u)$. Also we make the convention that ${\rm lk}(\emptyset) = \sV$. Notice that $\sU \cap  {\rm lk}(\sU) = \emptyset$.

Graph product groups naturally admit many amalgamated free product decompositions. One such decomposition which is essential for deriving our main results, involves full subgroup factors in   \cite[Lemma 3.20]{Gr90} as follows. For any $w \in \sV$ we have
\begin{equation}\label{afpdesc}\sG\{\Gamma_v\} = \Gamma_{\sV\setminus \{w\}} \ast_{ \Gamma_{\rm lk}(w)} \Gamma_{{\rm st}(w)},\end{equation}
where ${\rm st} (w) = \{w\} \cup {\rm lk} (w)$. Notice that $\Gamma_{{\rm lk}(w)}\lneqq \Gamma_{{\rm st}(w)}$ but it could be the case that $\Gamma_{{\rm lk}(w)}=\Gamma_{\sV\setminus \{w\}} $, when $\sV={\rm st}(w)$. In this case the amalgam decomposition is called degenerate.

Similarly for every subgraph $\sU\subset \sG$ we denote by ${\rm st}(\sU)= \sU \cup {\rm lk}(\sU )$. A maximal complete subgraph $\sC\subseteq \sG$ is called a {\it clique} and the collections of all cliques of $\sG$ will be denoted by  ${\rm cliq}(\sG)$. Below we highlight various properties of full subgroups that will be useful in this paper. 
The first one is \cite[Lemma 3.7]{AM10}, the second one is \cite[Proposition 3.13]{AM10}, while the third one is \cite[Proposition 3.4]{AM10}.

\begin{prop}\emph{\cite{AM10}}\label{proposition.AM10}
Let $\Gamma=\sG \{\Gamma_v\}$ be any graph product of groups, $g\in\Gamma$ and let $\sS,\sT\subseteq \sG$ be any subgraphs. Then the following hold.
\begin{enumerate}

     \item If $g \Gamma_{\sT} g^{-1}\subset \Gamma_\sS$, then there is $h\in\Gamma_{\sS}$ such that $g \Gamma_{\sT} g^{-1}= h \Gamma_{\sT\cap\sS} h^{-1}$. In particular, if $\sS=\sT$, then $g \Gamma_{\sT} g^{-1}= \Gamma_\sS$.
    
    \item The normalizer of $\Gamma_\sT$ inside $\Gamma$ satisfies $N_{\Gamma}(\Gamma_{\sT})=\Gamma_{\sT\cup {\rm link}(\sT)}$.
    
    \item There exist $\sD\subseteq \sS\cap \sT$ and $h\in \Gamma_\sT$ such that $g \Gamma_{\sS} g^{-1}\cap \Gamma_{\sT}=h \Gamma_{\sD} h^{-1}$.

\end{enumerate}

\end{prop}





\subsection {Popa's intertwining-by-bimodules techniques} We next recall from  \cite [Theorem 2.1, Corollary 2.3]{Po03} Popa's {\it intertwining-by-bimodules} technique, which is a powerful criterion for identifying intertwiners between arbitrary subalgebras of tracial von Neumann algebras.


\begin{thm}[\cite{Po03}]\label{corner} Let $(\M,\tau)$ be a tracial von Neumann algebra and $\P\subset p\M p,\Q\subset q\M q$  be von Neumann subalgebras. 
Then the following  are equivalent:

\begin{enumerate}

\item There exist projections $p_0\in \P, q_0\in \Q$, a $*$-homomorphism $\theta:p_0\P p_0\rightarrow q_0\Q q_0$  and a non-zero partial isometry $v\in q_0\M p_0$ such that $\theta(x)v=vx$, for all $x\in p_0\P p_0$.


\item There is no sequence $(u_n)_{n\ge 1}\subset\mathcal U(\P)$ satisfying $\|E_\Q(x^*u_ny)\|_2\rightarrow 0$, for all $x,y\in p\M$.
\end{enumerate}

If one of these equivalent conditions holds true,  we write $\P\prec_{\M}\Q$, and say that {\it a corner of $\P$ embeds into $\Q$ inside $\M$.}
Moreover, if $\P p'\prec_{\M}\Q$ for any non-zero projection $p'\in \P'\cap p\M p$, then we write $\P\prec^{s}_{\M}\Q$.
\end{thm}

Given an arbitrary graph product group,
our next lemma clarifies the intertwining of subalgebras of full subgroups in the associated graph product group von Neumann algebra.

\begin{lem}\label{corollary.graph.CI17}
Let $\Gamma=\sG \{\Gamma_v\}$ be any graph product of infinite groups  and let $\sS,\sT\subseteq \sG$ be any subgraphs. 
\noindent If $\L(\Gamma_{\sS})\prec_{\L(\Gamma)}  \L(\Gamma_\sT)$, then $\sS\subset\sT$.
\end{lem}

\begin{proof}
By applying \cite[Lemma 2.2]{CI17} there is $g\in\Gamma$ such that $[\Gamma_\sS:  \Gamma_\sS  \cap g \Gamma_\sT g^{-1}]<\infty$. 
By Proposition \ref{proposition.AM10} one can find a subgraph $\sP\subseteq \sS \cap \sT$ and $k\in \Gamma_\sS$ so that  $\Gamma_\sS\cap g \Gamma_\sT g^{-1}= k\Gamma_\sP k^{-1}$. Thus $k\Gamma_\sP k^{-1}< \Gamma_\sS$ is a finite index subgroup. Since $k\in \Gamma_\sS$, it follows that $\Gamma_\sP<\Gamma_\sS$ has finite index as well.
Since $|\Gamma_v| =\infty$, for all $v\in \sG$, we must have that
$[\Gamma_\sS: \Gamma_\sP]=1$, and hence,  $\Gamma_\sS=\Gamma_\sP$. Thus, $\sS = \sP\subset \sS \cap \sT$, and hence, $\sS \subset \sT$.
\end{proof}

\begin{rem}\label{remark.inclusion.subgraph}
The proof of Lemma \ref{corollary.graph.CI17} shows that if $\Gamma=\sG \{\Gamma_v\}$ is a graph product of infinite groups  and  $\sS,\sT\subseteq \sG$ are  subgraphs such that $[\Gamma_\sS: \Gamma_\sS\cap g \Gamma_\sT g^{-1}]<\infty$ for some $g\in\Gamma$, then $\sS\subseteq \sT$.
\end{rem}

\subsection{Quasinormalizers of von Neumann algebras}

\noindent Given an inclusion $\P \subset \M$ of tracial von Neumann algebras we define the quasi-normalizer  $\mathscr {QN}_{\mathcal M}(\mathcal P)$ as the subgroup of all elements $x\in \M$ for which there exist $x_1,...,x_n\in \M$ such that $\P x\subseteq \sum x_i \P$ and $x\P \subseteq \sum \P x_i$ (see \cite[Definition 4.8]{Po99}).

\begin{lem}[\cite{Po03,FGS10}]\label{QN1}
Let $\P\subset \M$ be tracial von Neumann algebras. For any projection $p\in\P$, we have that
     $W^*({\rm \mathscr{QN}_{p\M p}}(p\P p))=pW^*({\rm \mathscr{QN}_{\M}}(\P))p$. 

\end{lem}

\noindent Given a group inclusion $H<G$, the quasi-normalizer ${\rm QN}_G(H)$ is the group of all $g\in G$ for which exists a finite set $F\subset G$ such that $Hg\subset FH$ and $gH\subset HF$. The following result provides a relation between the group theoretical quasi-normalizer and the von Neumann algebraic one.

\begin{lem}[Corollary 5.2 in \cite{FGS10}]\label{QN2}
Let $\Lambda<\Gamma$ be countable groups. Then 
    $W^*(\mathscr{QN}_{\L(\Gamma)}(\L(\Lambda)))=\L({\rm QN}_\Gamma(\Lambda))$.

\end{lem}

\noindent
We continue by computing the quasi-normalizer of 
subalgebras of full subgroups in any graph product group von Neumann algebra. More generally, we show the following.

\begin{thm}\label{controlquasinormalizer1}
Let $\Gamma=\sG \{\Gamma_v\}$ be any graph product of infinite groups  and let $\sS,\sT\subseteq \sG$ be any subgraphs. Denote by $\M= \L(\Gamma)$ and assume there exist $x,x_1,x_2,...,x_n \in \M$ such that $\L(\Gamma_\sS)x\subseteq \sum^n_{k=1} x_k \L(\Gamma_\sT)$. Thus $\sS\subseteq \sT$ and $x\in \L(\Gamma_{\sT \cup {\rm lk}(\sS) })$.     
\end{thm}
\begin{proof} Using the proof of \cite[Lemma 2.8]{CI17} and  \cite[Claim 2.3]{CI17} we obtain that $x$ belongs to the $\|\cdot\|_2$-closure of the linear span of $\{u_g\}_{g\in S}$. Here, $S$ denotes the set of all elements $g\in \Gamma$ for which $[\Gamma_\sS: \Gamma_\sS\cap g \Gamma_\sT g^{-1}]<\infty$. By assuming that $x\neq 0$, it follows that $S$ is non-empty. Fix $g\in S$. By using Remark \ref{remark.inclusion.subgraph}, we derive that
 $\sS \subseteq \sT$, which gives the first part of the conclusion. 

For proving the second part, note that by Proposition \ref{proposition.AM10} one can find a subgraph $\sP\subseteq \sS$ and $k\in \Gamma_\sS$ so that  $\Gamma_\sS\cap g \Gamma_\sT g^{-1}= k\Gamma_\sP k^{-1}$. Thus $k\Gamma_\sP k^{-1}< \Gamma_\sS$ is a finite index subgroup.  
  Since $k\in \Gamma_\sS$ this further implies that $\Gamma_\sP<\Gamma_\sS$ has finite index, and hence $\sP=\sS$. Using again that $k\in\Gamma_\sS$, we get
    $\Gamma_\sS\cap g \Gamma_\sT g^{-1}= k\Gamma_\sP k^{-1}= \Gamma_\sS$, and thus, 
$ g^{-1}\Gamma_\sS g<  \Gamma_\sT $. By Proposition \ref{proposition.AM10} one can find $r\in \Gamma_\sT$ such that $g^{-1}\Gamma_\sS g= r  \Gamma_\sS r^{-1}$. This relation implies in particular that  $gr \in N_\Gamma(\Gamma_\sS)$ and since  $N_\Gamma(\Gamma_\sS)= \Gamma_{\sS\sqcup {\rm lk} (\sS )}$ (see Proposition \ref{proposition.AM10}),  we conclude that $gr\in  \Gamma_{\sS\cup {\rm lk} (\sS )}$. Therefore, $g\in  \Gamma_{\sS\cup {\rm lk} (\sS )}\Gamma_\sT\subset \Gamma_{\sT \cup {\rm lk}(\sS)}$. This gives the desired conclusion.
\end{proof}

\begin{cor}\label{controlquasinormalizer2}Let $\Gamma=\sG \{\Gamma_v\}$ be any graph product of infinite groups  and let $\sC\in {\rm cliq}( \sG)$ be a clique with at least two vertices. Fix a vertex $v\in \sC$ such that ${\rm lk}(\sC \setminus \{v\})=\{v\}$.  Denote by $\M= \L(\Gamma)$ and assume there exist $x,x_1,x_2,...,x_n \in \M$ such that $\L(\Gamma_{\sC \setminus \{v\}})x\subseteq \sum^n_{k=1} x_k \L(\Gamma_\sC)$. Then  $x\in \L(\Gamma_{\sC })$. 

\begin{proof} Follows applying Theorem \ref{controlquasinormalizer1} for $\sS =\sC \setminus \{v\} $ and $\sT =\sC$.\end{proof}

\end{cor}

\begin{lem}\label{lemma.control.qn}\label{controlonesidednorm3}
Let $\Gamma=\sG \{\Gamma_v\}$ be a graph product of groups and let $\sC\in {\rm cliq}(\sG)$ be a clique. Let $\P\subset p\L(\Gamma_{\sC})p$ be a von Neumann subalgebra such that $\P\nprec_{\L(\Gamma_\sC)} L(\Gamma_{\sC_{\hat v}})$, for any $v\in \sC$.

\noindent If $x\in \L(\Gamma)$ satisfies $x\P\subset \sum_{i=1}^n \L(\Gamma_{\sC})x_i$ for some $x_1,\dots,x_n\in \L(\Gamma)$, then $xp\in \L(\Gamma_\sC)$.
\end{lem}

\begin{proof}
Let $g\in \Gamma\setminus \Gamma_\sC$. By using Proposition \ref{proposition.AM10} there exist $h\in \Gamma_\sC$ and $\sD\subset \sC$ such that $\Gamma_\sC\cap  g \Gamma_\sC g^{-1}=h \Gamma_\sD h^{-1}$.  
Note that Theorem \ref{controlquasinormalizer1} shows $\rm{QN}_\Gamma^{(1)}(\Gamma_\sC)=\Gamma_\sC$ and therefore $\sD\neq\sC$; otherwise, we would get that $g\in \rm{QN}_\Gamma^{(1)}(\Gamma_\sC)=\Gamma_\sC$, contradiction. Thus, we deduce from the assumption that $\P\nprec_{\L(\Gamma_\sC)} L(\Gamma_{\sC}\cap g \Gamma_{\sC} g^{-1})$ for any $g\in \Gamma\setminus \Gamma_\sC$. The conclusion now follows from \cite[Lemma 2.7]{CI17}.
\end{proof}

\subsection{A result on normalizers in tensor product factors}

Our next proposition describes the normalizer of a II$_1$ factor $\N$ inside the tensor product of $\N$ with another II$_1$ factor. More precisely, we have:

\begin{prop}\label{productunitaries} Let $\N$ and $\P$ be II$_1$ factors and denote by $\M=\N\bar\otimes \P$. If $u\in \sU (\M)$ satisfies $u\N u^*=\N$, then one can find $a\in \sU(\N)$ and $b\in \sU(\P)$ such that $u=a\otimes b$. 

\end{prop}

\begin{proof} Let $(\xi_i)_{i\in I}\subset L^2(\P)$ be a Pimsner-Popa basis for the inclusion $\N\subset \M$, let  $u= \sum_{i} E_\N(u \xi_i^*) \otimes \xi_i$ and denote by $\eta_i = E_\N(u \xi_i^*)$.
If $\theta\,:\, \N \ra \N$ denotes the $\ast$-isomorphism  $\theta={\rm ad}(u) $ then we have $\theta(x) u =ux$ for all $x\in \N$. This combined with the above formula yields $\theta(x) \eta_i \otimes \xi_i = \theta(x)u=ux= \eta_i x \otimes \xi_i $. Hence for all $x\in \N$ and all $i$ we have \begin{equation}\label{intertwiningell2}
    \theta(x)\eta_i=\eta_i x. 
\end{equation}
Let  $u_i\in \N$ be the partial isometry in the polar decomposition of $\eta_i$. Thus $ \theta(x)u_i =u_i x$ for all $x\in \N$ and all $i$. In particular we get  $u^*_iu_i\in \N'\cap \N =\mathbb C 1$ and hence $u_i \in \sU(\N )$ for all $i$. The prior relations also imply that $u_i^*x u_i =\theta(x)= u_j^*xu_j$ for all $i,j\in I$. In particular, we have that $u_iu_j^*\in \N'\cap \N=\mathbb C 1$ and thus one can find scalars $c_{i,j}\in \mathbb T$ so that $u_i = c_{i,j} u_j$ for all $i,j\in I$. Relation \eqref{intertwiningell2} also implies that $|\eta_i|\in \N'\cap L^2(\N)$ and since $\N$ is a II$_1$ factor we get that $|\eta_i|\in \mathbb C 1$. In conclusion, $\eta_i \in \mathbb C \sU(\N)$ for all $i$ and one can find  $d_{i,j}\in \mathbb C$ such that $\eta_i = d_{i,j}\eta_j$ for all $i,j \in I$. Fix $j\in I$ with $\eta_j\neq 0$. Using the above relations we have that $u = \sum_i\eta_i\otimes \xi_i = \sum_i d_{i,j} \eta_j \otimes \xi_i = \eta_j \otimes (\sum_i d_{i,j} \xi_i)= \eta_j \otimes b $, where we denoted by $b= \sum_i d_{i,j}  \xi_i \in L^2(\P)$. Since $\eta_j \in \mathbb C \sU(\N)$ we get the desired conclusion.
\end{proof}

\section{Wreath-like product groups}

In the previous work \cite{CIOS21} it was introduced a new category of groups called \emph{wreath-like product groups}.  To briefly recall their construction let $A$ and $B$ be any countable groups and assume that $B \ca I$ is an action on a countable set. One says $W$ is a wreath-like product of $A$ and $B\ca I$  if it can be realized as a group  extension 

\begin{equation}\label{regwreathlike1'''}
    1 \ra \bigoplus_{i\in I} A_i\hookrightarrow W \overset{\varepsilon}{\twoheadrightarrow} B \ra 1 
 \end{equation}
 which satisfies the following properties:
 \begin{enumerate}
     \item [a)] $A_i\cong A$ for all $i\in I$, and 
     \item [b)] the action by conjugation of  $W$ on $\bigoplus_{i\in I} A_i$  permutes the direct summands according to the rule \begin{equation*}w A_i w^{-1}= A_{\varepsilon(w)i}\text{ for all }w\in W, i\in I.\end{equation*}
 \end{enumerate}
 The class of all such wreath-like groups is denoted by $\W\R(A, B\ca I)$. When $I= B$ and the action $B\ca I$ is by translation this consists of so-called regular wreath-like product groups and we simply denote their class by $\W\R(A, B)$.
 
 Notice that every classical generalized wreath product $A\wr_I B \in \W\R (A,B\ca I)$. However, building examples of non-split wreath-like products is a far more involved problem. One way to approach this is through the use of the so-called Magnus embbeding, \cite{Ma37}; these are quotients groups of the form $\Gamma/[\La,\La]$ where $\La\lhd \Gamma$ is a normal subgroup. Methods of these type were used by Cohen-Lyndon to produce many such quotients in the context of one-relator groups. 
 The following result is a particular case of \cite[Corollary 4.6]{CIOS21} and relies on the prior works \cite{Osi07,DGO11,Su20}.

 \begin{cor}\label{Cor:WRDFComm}
Let $G$ be an icc hyperbolic group. For every infinite order element $g\in G$, there exists $d\in \mathbb N$ such that for every $k\in \mathbb N$ divisible by $d$ we have the following.
\begin{enumerate}
\item[(a)] $G/[\ll g^{k}\rr, \ll g^{k}\rr]\in \W\R (\mathbb Z, G/\ll g^k\rr \curvearrowright I)$, where $\langle g^k\rangle$ is normal in $E_G(g)$, the action $G/\ll g^k\rr\curvearrowright I$ is transitive, and all the stabilisers of elements of $I$ are isomorphic to the finite group $E_G(g)/\langle g^k\rangle$. Here, $E_G(g)$ denotes the elementary subgroup generated by g, $\langle g^k \rangle$ denotes the subgroup generated by $g^k$ and $\ll g^k \rr$ denotes  the smallest normal subgroup that contains $g^k$.
\item[(b)] $G/\ll g^k\rr$ is an icc hyperbolic group.
\end{enumerate}
\end{cor}

 Developing a new quotienting method in the context of Cohen-Lyndon triples in \cite[Theorem 2.5]{CIOS21} were constructed many examples of property (T) regular wreath-like product groups as follows:

 \begin{thm}[\text{\cite{CIOS21}}]\label{AHQ}
Let $G$ be a hyperbolic group. For every finitely generated group $A$, there exists a quotient $W$ of $G$ such that $W\in \W\R (A,B)$ for some hyperbolic group $B$.
\end{thm}

 For further use we also recall the following result on prescribed outer automorphisms of property (T) regular wreath-like product groups that was established in \cite[Theorem 6.9]{CIOS21}.

\begin{thm}[\text{\cite{CIOS21}}]\label{Thm:JC}
For every finitely presented group $Q$ and every finitely generated group $A_0$, there exist groups $A$, $B$ and a regular wreath-like product $W\in \W\R(A,B)$ with the following properties.
\begin{enumerate}
\item[(a)] $W$ has property (T) and has no non-trivial characters.

\item[(b)] $A$ is the direct sum of $|Q|$ copies of $A_0$. In particular, $A=A_0$ if $Q=\{ 1\}$.

\item[(c)] $B$ is an ICC normal subgroup of a hyperbolic group $H$ and $H/B\cong Q$. In particular, $B$ is hyperbolic whenever $Q$ is finite.
\item[(d)] $Out(W)\cong Q$.
\end{enumerate}
\end{thm}

\begin{rem}Since $A_0$ can be any finitely generated group it follows that if we fix the group $Q$ there are infinitely many pairwise nonisomorphic regular wreath like product groups $W\in \W\R(A,B)$ which satisfy (a)-(d) in the prior theorem.  
\end{rem}

\subsection{Unique prime factorization for von Neumann algebras of wreath-like product groups}

In this subsection, more precisely in Theorem \ref{theorem.UPF.WR}, we show that von Neumann algebras of certain wreath-like product groups satisfy Ozawa and Popa's unique prime factorization \cite{OP07}. First, we point out the following structural result for commuting property (T) von Neumann subalgebras of von Neumann algebras that arise from trace preserving actions of certain wreath-like product groups. 

\begin{lem}\label{lemma.commuting.subalgebras}
Let $\Gamma$ be a wreath-like product group of the form $\W\R (A, B\ca I)$, where $A$ is abelian, $B$ is an icc  subgroup of  a hyperbolic group. Let $\Gamma\curvearrowright \N$ be a trace preserving action and denote $\M=\N\rtimes\Gamma$.

If $\A, \B\subset p\M  p$ are commuting property (T) von Neumann subalgebras, then $\A\prec_\M \N$ or $\B\prec_\M \N$.
\end{lem}

 The proof of Lemma \ref{lemma.commuting.subalgebras} follows from \cite[Theorem 7.15]{CIOS21} and the main ingredient of its proof is  Popa and Vaes’ structure theorem for normalizers in crossed
products arising from actions of hyperbolic group \cite{PV12}.

\begin{thm}\label{theorem.UPF.WR}
For any $i\in\overline{1,n}$, let $\Gamma_i$ be a property (T) wreath-like product group of the form $\W\R (A, B\ca I)$, where $A$ is abelian and $B$ is an icc  subgroup of  a hyperbolic group.

If $\M:=\L(\Gamma_1)\bar\otimes\dots\bar\otimes \L(\Gamma_n)=\P_1\bar\otimes \P_2$ is a tensor product 
decomposition into II$_1$ factors, then there exist a unitary $u\in \M$, a decomposition $\M=\P_1^t\bar\otimes \P_2^{1/t}$ and a partition $T_1\sqcup T_2=\overline{1,n}$ such that 
$\L(\times_{k\in S_j} \Gamma_k)=u \P_j^{t_j} u^*$, for any $j\in\{1,2\}.$
\end{thm}

\begin{proof}
To fix some notation, we have that for any $i\in\overline{1,n}$, $\Gamma_i$ belongs to
$\W\R (A_i, B_i \ca I_i)$, where $A_i$ is abelian, $B_i$ is an icc  subgroup of  a hyperbolic group.  
Note that since $\P_1$ and $\P_2$ have property (T), by applying Lemma \ref{lemma.commuting.subalgebras}, we obtain a map $\phi:\overline{1,n}\to\overline{1,2}$ such that $\P_{\phi(i)}\prec_\M \bar\otimes_{k\neq i} \L(\Gamma_k)$, for any $i\in\overline{1,n}$. By  \cite[Lemma 2.8(2)]{DHI16} there exists a partition $\overline{1,n}=S_1\sqcup S_2$ such that $\P_j\prec_\M \L(\times_{k\in S_j} \Gamma_k)$, for any $j$. By passing to relative commutants, we get that $\L(\times_{k\in S_j} \Gamma_k)\prec_\M \P_j$, for any $j$. The conclusion of the theorem follows now by using standards arguments that rely on \cite[Proposition 12]{OP04} and \cite[Theorem A]{Ge95}.
\end{proof}

\begin{cor}\label{corollary.wr.rigidity}
Let $\Gamma_1,\dots,\Gamma_n$ and $\Lambda_1\dots,\Lambda_m$ be property (T) wreath-like product groups of the form $\W\R (A, B\ca I)$, where $A$ is abelian, $B$ is an icc  subgroup of  a hyperbolic group, and $B \curvearrowright I$ has infinite stabilizers.    

If there exists $t>0$ such that $\L(\Gamma_1\times\dots\times\Gamma_n)^t=\L(\Lambda_1\times\dots\times\Lambda_m)$, then $t=1$, $n=m$ and there is a unitary $u\in \L(\Gamma_1\times\dots\times\Gamma_n)$ such that  $u\mathbb T (\Gamma_1\times\dots\times\Gamma_n)u^*= \mathbb T (\Lambda_1\times\dots\times\Lambda_m)$.
\end{cor}

\begin{proof}
    The result follows directly by combining Theorem \ref{theorem.UPF.WR} and \cite[Theorem 8.4]{CIOS21}.
\end{proof}

\section{Graph product groups  associated with cycles of cliques graphs}\label{CC1} In this section we highlight a class of graphs considered to have good clustering properties. Specifically, a graph   $\sG$ is called a \emph{simple cycle of cliques} (and it said to belong to class ${\rm CC}$) if there is an enumeration of its cliques set $ {\rm cliq}(\sG)= \{\sC_1,..., \sC_n\} $ with $n\geq 4$ such that the subgraphs $\sC_{i,j} :=\sC_i \cap \sC_j$ satisfy the following conditions: \begin{equation*}
    \sC_{i,j}=\begin{cases}\emptyset, \text{ if }  \hat i-\hat j\in \mathbb Z_n\setminus \{\hat 1, \widehat {n-1}\}\\
\neq \emptyset, \text{ if }  \hat i-\hat j\in \{ \hat 1, \widehat {n-1}\}.
\end{cases}
  \end{equation*}
Here, the classes $\hat i, \hat j\in \mathbb Z/n\mathbb Z$.  We will also refer to $ {\rm cliq}(\sG)= \{\sC_1, ..., \sC_n\} $ satisfying the previous properties as the \emph{the consecutive enumeration} of the cliques of $\sG$.   

For every $i\in\overline{1,n}$ we denote  $\sC^{\rm int}_i := \sC_i \setminus  \left(\sC_{i-1,i} \cup \sC_{i,i+1}\right) $, where we declare that $0=n$ and  $n+1=1$. When  $\sC^{\rm int}_i\neq \emptyset$ for all $i\in\overline{1,n}$ one says that $\sG$ belongs to class ${\rm CC}_1$. Most of our main results will involve graphs of these form. Throughout this article we will use all these notations consistently.

A basic example of a graph in the class CC$_1$ is a simple cycle of triangles  called $\mathcal T_n $, where $n$ is the number of cliques; see picture below for $n=16$.

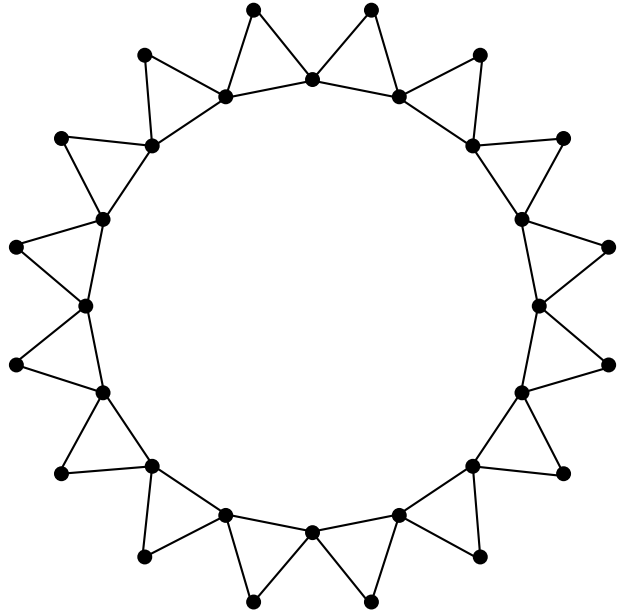
\begin{equation}
\begin{tikzpicture}

\node[thick, draw,minimum size={2*3cm},regular polygon,regular polygon
 sides=16,rotate=11.25] at
 (3,3) (16-gon)[text=blue] {};

\node[thick, draw=white,minimum size={2*4cm},regular polygon,regular polygon
 sides=16] at
 (3,3) (sixteengon)[text=blue] {};

\node[thick, draw=white,minimum size={2*4cm},regular polygon,regular polygon
 sides=16, rotate=22] at
 (3,3) (6teengon)[text=blue] {};

  \foreach \p [count=\n] in {1,...,16}{
 \node [fill=black,circle,inner sep=2pt] at (16-gon.corner \n){};
 \node [fill=black,circle,inner sep=2pt] at (sixteengon.corner \n){};
 \node [fill=black,circle,inner sep=0pt] at (6teengon.corner \n){};
 \draw[thick] (sixteengon.corner \n) -- (16-gon.corner \n) -- (6teengon.corner \n);
}

\end{tikzpicture}
\end{equation}

In fact every graph $\sG\in CC_1$ appears as a two-level clustered graph which is a specific retraction of $\mathcal T_n$ as follows.
There exists a graph projection map $\Phi: \mathcal G\ra \T_n$ such that for every vertex $v \in \T_n$ the cluster $\Phi^{-1}(v)\subset \sG$ is a complete subgraph of $\sG$. In addition, whenever $v,w \in \mathcal T_n$ are connected in $\mathcal T_n$ there is an interconnection edge in $\sG$ between all vertices of the corresponding clusters $\Phi^{-1}(v)$ and $\Phi^{-1}(w)$.

We continue by recording some elementary combinatorial properties of graph product groups associated with graphs that are simple cycles of cliques. The proof of the following lemma is straightforward and we leave it to the reader.

\begin{lem}\label{altgpdecomp} Let $\sG\in {\rm CC}_1$ and let $\sC_1,...,\sC_n$ be an enumeration of its consecutive cliques. Let $\{\Gamma_v , v\in \sV\}$ be a collection of groups and let $\Gamma_\sG$ be the corresponding graph product group. We denote by $\{w_i \}_{i=1}^n$ the petal outer vertices of $\mathcal T_n$ and by $\{b_i\}_{i=1}^n$ the petal base vertices of $\mathcal T_n$. 

\noindent
Then $\Gamma_\sG$ can be realized as a graph product $\Gamma'_{\mathcal T_n}$ associated to the graph $\mathcal T_n$, where the vertex groups are defined by $\Gamma'_{w_i}= \oplus_{v\in \sC^{\rm int}_i} \Gamma_v$ and $\Gamma'_{b_i}= \oplus_{v\in \sC_{i-1,i}}\Gamma_v$ for every $i\in \overline{1,n}$.  
\end{lem}

\begin{prop}\label{combfacts} Let $\sG\in{\rm CC}_1$ and let $\sC_1,...,\sC_n$ be an enumeration of its consecutive cliques. Let $\{\Gamma_v , v\in \sV\}$ be a collection of infinite groups and let $\Gamma_\sG$ be the corresponding graph product group.  Then the following properties hold:\begin{enumerate}
    \item If  $g\in \Gamma_{\sC_{i-1}\triangle \sC_i}$ and $h\in \Gamma_{\sC_{i}\triangle \sC_{i+1}}$ satisfy $gh\in  \Gamma_{\sG\setminus \sC^{\rm int}_i} $ then one can find $a\in \Gamma_{(\sC_{i-1}\setminus \sC_{i-1,i}) \cup \sC_{i,i+1}}$, $s\in \Gamma_{\sC^{\rm int}_i}$
and $b\in \Gamma_{(\sC_{i+1}\setminus \sC_{i,i+1})\cup \sC_{i-1,i}}$ such that $g= as$ and $h = s^{-1}b$. 

\item Let $g\in \Gamma_{\sC_{i-2,i-1} \cup \sC_{i,i+1} }, h\in \Gamma_{\sC_{i-1,i} \cup \sC_{i+1,i+2} }, k\in \Gamma_{\sC_{i,i+1} \cup \sC_{i+2,i+3} }$ such that  $ghk\in \Gamma_{\sC_{i-2,i-1} \cup \sC_{i-1,i}  \cup \sC_{i+1,i+2}  \cup \sC_{i+2,i+3} }$. Thus one can find $a\in \Gamma_{\sC_{i-2,i-1} }$, $b\in \Gamma_{\sC_{i+2,i+3} }$ and $s \in \Gamma_{\sC_{i,i+1} }$ such that $g= as$  and $k = s^{-1}b$.
\item For each $ i\in \overline{1,n}$ let $x_{i,i+1}\in \Gamma_{\sC_i \cup \sC_{i+1}}$ such that $x_{1,2} x_{2,3}\cdots x_{n-1,n}x_{n,1}=1$. Then for each $ i\in \overline{1,n}$ one can find $a_i \in \Gamma_{\sC_{i-1,i}}$, $b_i \in \Gamma_{\sC^{\rm int}_i}$, $c_i \in \Gamma_{\sC_{i,i+1}}$  such that $x_{i,i+1}= a_i b_i c_i b^{-1}_{i+1} a^{-1}_{i+2}c^{-1}_{i+1}$. Here we convene that $n+1=1$, $n+2=2$, etc.
\end{enumerate} 

\end{prop}

\begin{proof}
Here $\Delta$ is the symmetric difference operation defined by $A\Delta B=(A\cup B)\setminus (A\cap B)$. We recall the {\it normal form} \cite[Theorem 3.9]{Gr90}, which in graph product groups plays the role that reduced words play in free product groups. If $1\neq g\in \Gamma_\sG$ is expressed as $g=g_{1}\cdots g_{n}$, we say $g$ is in normal form if each $g_{i}$ is a non-identity element of some vertex group (called a {\it syllable}) and if it is impossible, through repeated swapping of syllables (corresponding to adjacent vertices in $\sG$), to bring together two syllables from the same vertex group. By \cite[Theorem 3.9]{Gr90}, every $1\neq g\in \Gamma_\sG$ has a normal form $g=g_{1}\cdots g_{n}$ and it is unique up to a finite number of consecutive syllable shuffles. Moreover, given any sequence of syllables $g_{1}\cdots g_{n}$, there is an inductive procedure for putting this sequence into normal form: if $h_{1}\cdots h_{r}$ is the normal form of $g_{1}\cdots g_{m}$, then the normal form of $g_{1}\cdots g_{m+1}$ is either ($i$) $h_{1}\cdots h_{r}$ if $g_{m+1}=1$, ($ii$) $h_{1}\cdots h_{j-1}h_{j+1}\cdots h_{r}$ if $h_{j}$ shuffles to the end and $g_{m+1}=h_{j}^{-1}$, ($iii$) $h_{1}\cdots h_{j-1}h_{j+1}\cdots h_{r}(h_{j}g_{m+1})$ if $h_{j}$ shuffles to the end, $g_{m+1}\neq h_{j}^{-1}$ and $g_{m+1},h_{j}$ belong to the same vertex group, or ($iv$) $h_{1}\cdots h_{r}g_{m+1}$ if $g_{m+1}$ is in a different vertex group from that of every syllable which can be shuffled down. Note that the normal form of an element $g\in\Gamma_\sG$ has minimal syllable length with respect to all the sequences of syllables representing $g$.  
 
 \indent We are now ready to prove the three assertions of the proposition. For 1., let $g=g_{1}\cdots g_{n}$ and $h=h_{1}\cdots h_{m}$ be the normal forms of $g$ and $h$. Then $gh$ has a normal form $gh=k_{1}\cdots k_{r}$, determined by the procedure described in the previous paragraph.  By assumption, $k_{j}\notin\bigcup_{v\in\sC^{\rm int}_i}\{\Gamma_{v}\}$ for all $j\in\overline{1,r}$.  Now, if $g_{j}\notin\bigcup_{v\in\sC^{\rm int}_i}\{\Gamma_{v}\}$ for all $j\in\overline{1,n}$, then each $h_{i}\in\bigcup_{v\in\sC^{\rm int}_i}\{\Gamma_{v}\}$ is one of the syllables occurring in the normal form of $gh$. Since this cannot happen we have $h_{i}\notin\bigcup_{v\in\sC^{\rm int}_i}\{\Gamma_{v}\}$ for all $i\in\overline{1,m}$, and hence we can take $a=g$, $b=h$, and $s=\text{empty word}$. Assume $g_{j}\in\bigcup_{v\in\sC^{\rm int}_i}\{\Gamma_{v}\}$ for some $j\in\overline{1,n}$. Notice that we may assume that $j=n$ since if $g_{i}\in \bigcup_{v\in\sC_{i-1}\setminus\sC_{i}}\{\Gamma_{v}\}$ for some $i\in\overline{j+1,n}$ then $g_{j}$ would be a syllable in $gh$ since it cannot be shuffled past $g_{i}$, which shows that $g_{j+1}\cdots g_{n}\in\bigcup_{v\in\sC_{i}}\{\Gamma_{v}\}$. This implies that $g_{n}^{-1}=h_{i}$ for some $i\in\overline{1,m}$. Choosing the smallest such $i$ and noting that $h_{1},...,h_{i-1}\in\bigcup_{v\in\sC_{i}}\{\Gamma_{v}\}$ (since it must be possible to shuffle $h_{i}$ up to $g_{n}$ as in ($ii$) of the previous paragraph), we may assume that $h_{1}=g_{n}^{-1}$. Continuing in this way we see that we can take $a=g_{1}\cdots g_{k-1}$, $b=h_{n-k+2}\cdots h_{m}$, and $s=g_{k}\cdots g_{n}$ where $g_{j}\notin \bigcup_{v\in\sC^{\rm int}_i}\{\Gamma_{v}\}$ for all $j\in \overline{1,k-1}$ and $h=g_{n}
^{-1},...,g_{k}^{-1}h_{n-k+2},...,h_{m}$. Notice too that none of the syllables  $h_{n-k+2},...,h_{m}$ can belong to $\bigcup_{v\in\sC^{\rm int}_i}\{\Gamma_{v}\}$, since the inverse of such a syllable cannot be any of the syllables $g_{1},...,g_{k-1}$. This  proves ($1$).

\indent For 2., let $g=g_{1}\cdots g_{n}$, $h=h_{1}\cdots h_{m}$ and $k=k_{1}\cdots k_{r}$ be normal forms. If $g_{i}\notin\bigcup_{v\in\sC_{i,i+1}}\{\Gamma_{v}\}$ for all $i\in\overline{1,n}$, then $k_{j}\notin\bigcup_{v\in\sC_{i,i+1}}\{\Gamma_{v}\}$ for all $j\in\overline{1,n}$  since neither $h$ nor $ghk$ have normal forms with syllables in $\bigcup_{v\in\sC_{i,i+1}}\{\Gamma_{v}\}$ by assumption, and hence we can take $a=g$, $b=k$ and $s=\text{empty word}$. Otherwise we must have $g_{j}\in \Gamma_{v}$ for some $v\in\sC_{i,i+1}$, and as in the proof of part (1) we can assume $j=n$ and $k_{1}=g_{n}^{-1}$ (note that $g_{j}$ commutes with each syllable in the normal form of $h$). Continuing, we see that we can take $a=g_{1}\cdots g_{l-1}$, $b=k_{n-l+2}\cdots k_{r}$, and $s=g_{l}\cdots g_{n}$, where $g_{j}\notin\bigcup_{v\in\sC_{i,i+1}}\{\Gamma_{v}\}$ for all $j\in \overline{1,l-1}$  and $k=g_{n}
^{-1}\cdots g_{l}^{-1}k_{n-l+2}\cdots k_{r}$. This proves ($2$).

\indent For 3., observe first that every  $x_{i,i+1}\in \Gamma_{\sC_i\cup\sC_{i+1}}=\Gamma_{\sC_{i,i+1}} \times (\Gamma_{\sC_i\setminus \sC_{i+1}}*\Gamma_{\sC_{i+1}\setminus \sC_{i}})$ can be written in the form $\widetilde{a_{i}}\widetilde{b_{i}}\widetilde{c_{i}}\widetilde{d_{i}}\widetilde{e_{i}}\widetilde{f_{i}}$ where $\widetilde{a_{i}}\in \Gamma_{\sC_{i-1,i}}$, $\widetilde{b_{i}}\in \Gamma_{\sC^{\rm int}_i}$, $\widetilde{c_{i}}\in \Gamma_{\sC_{i,i+1}}$, $\widetilde{d_{i}}\in \Gamma_{\sC^{\rm int}_{i+1}}$, $\widetilde{e_{i}}\in \Gamma_{\sC_{i+1,i+2}}$, $\widetilde{f_{i}}\in\Gamma_{\sC_i \cup \sC_{i+1}}$. Moreover, we can assume that the normal form of $x_{i,i+1}$ is the sequence obtained by concatenating the normal forms of $\widetilde{a_{i}},\widetilde{b_{i}},\widetilde{c_{i}},\widetilde{d_{i}},\widetilde{e_{i}},\widetilde{f_{i}}$ and if $\widetilde{f_{i}}=f_{1}\cdots f_{n}$ is the normal form of $\widetilde{f_{i}}$, then $f_{1}$ belongs to a group $\Gamma_{v}$ where $v$ is vertex in $\sC_{i}\setminus\sC_{i+1}$. 

We continue by showing that we can assume that $\widetilde f_i=1$. Notice that this is the case if there is no syllable $g$ occurring in the normal form of $x_{i,i+1}$ belonging to $\bigcup_{v\in\sC_{i+1}\setminus\sC_{i}}\{\Gamma_{v}\}$; indeed, in this case $\widetilde{d_{i}},\widetilde{e_{i}}=1$ and all the syllables occurring in the normal form of $\widetilde{f_{i}}$ can be shuffled up to the normal forms of $\widetilde{a_{i}},\widetilde{b_{i}},\widetilde{c_{i}}$.  So it remains to assume that there is such a syllable $g$ and  assume by contradiction that $\widetilde f_i\neq 1.$ Notice that our hypotheses imply that $f_{1}^{-1}$ is a syllable in the normal form of $x_{i-2,i-1}x_{i-1,i}$ and $g^{-1}$ is a syllable in the normal form of $x_{i+1,i+2}x_{i+2,i+3}$. This implies that the normal form of $x_{1,2} x_{2,3}\cdots x_{i-1,i}x_{i,i+1}$ must still contain $f_{1}$ as $f_{1}^{-1}$ cannot shuffle past $g$ to cancel with $f_{1}$. Consequently,  the normal form of $x_{1,2} x_{2,3}\cdots x_{n-1,n}x_{n,1}$ must still contain $f_1$ as  $f_{1}^{-1}$ cannot shuffle past $g$ or $g^{-1}$ to cancel with $f_{1}$  
This gives a contradiction, and hence, we can assume $\widetilde{f}_{i}=1$.  

Next, we observe that $\widetilde{b_{i}}=\widetilde{d_{i-1}}^{-1}$ for each $i$ since our hypotheses imply that all the syllables occurring in the normal form of $\widetilde{b_{i}}^{-1}$ occur in the one for $x_{i-1,i}$ and only $\widetilde{d_{i-1}}$ has normal form with syllables coming from $\bigcup_{v\in\sC^{\rm int}_i}\{\Gamma_{v}\}$. To finish the proof, set $a_{i}=\widetilde{a_{i}}$, $b_{i}=\widetilde{b_{i}}$, and $c_{i}=\widetilde{c_{i}}$ and note that since $\widetilde{e_{i}}\widetilde{c_{i+1}}\widetilde{a_{i+2}}=1$ (being the only elements in our decompositions belonging to $\Gamma_{\sC_{i+1,i+2}}$) we have $x_{i,i+1}=\widetilde{a_{i}}\widetilde{b_{i}}\widetilde{c_{i}}\widetilde{d_{i}}\widetilde{e_{i}}=a_{i}b_{i}c_{i}b_{i+1}^{-1}a_{i+2}^{-1}(a_{i+2}\widetilde{e_{i}})=a_{i}b_{i}c_{i}b_{i+1}^{-1}a_{i+2}^{-1}c_{i+1}^{-1}$.      \end{proof}

\begin{lem}\label{trivial.amenable.radical}
Let $\Gamma= \mathscr G \{\Gamma_v,v\in \sV\}$ be a graph product of groups such that $\mathscr G\in {\rm CC}_1$. Then $\Gamma$ has trivial amenable radical.
\end{lem}

\begin{proof}
Assume by contradiction that there exists a non-trivial amenable normal subgroup $A$ of $\Gamma$. Since $\Gamma$ is icc, we get that $A$ is an infinite group. For any $w\in \sV$ note that ${\rm st}(w)\neq \sV$ and $\sG \{\Gamma_v\}=\Gamma_{\sV\setminus\{w\}}*_{\Gamma_{{\rm lk}(w)}} \Gamma_{{\rm st}(w)}$. Since $A$ is an amenable, normal subgroup of $\Gamma$, it follows from \cite[Theorem A]{Va13} that $\L(A)\prec_{\L(\Gamma)} \L(\Gamma_{{\rm lk}(w)})$. In particular, by using \cite[Lemma 2.4]{DHI16}, it follows that $\L(A)\prec^s_{\L(\Gamma)} \L(\Gamma_{\sC})$ for any $\sC\in {\rm cliq}(\sG)$. Let $\sC,\sD\in {\rm cliq}(\sG)$ such that $\sC\cap \sD=\emptyset$. Using \cite[Lemma 2.7]{Va10a}, there is $g\in\Gamma$ such that $\L(A)\prec^s_{\L(\Gamma)} \L(\Gamma_{\sC}\cap g \Gamma_{\sD} g^{-1})$. Note however that Proposition \ref{proposition.AM10} implies that $\Gamma_{\sC}\cap g \Gamma_{\sD} g^{-1}=1$. This shows that $\L(A)\prec_{\L(\Gamma)} \mathbb C 1$, and thus,
gives the contradiction that
$A$ is finite. 
\end{proof}

We end this section by recording a result describing all automorphisms of graph product groups $G_\sG$ associated with graphs in the class $\sG$. This is a particular case of a powerful theorem in geometric group theory established recently by Genevois and Martin \cite[Corollary C]{GM19}. 

To state the result we briefly recall a special class of automorphisms of graph product groups.  For any isometry $\sigma: \sG \ra \sG$ and any collection of group isomorphisms $\Phi=\{ \phi_v: \Gamma_v \ra \Gamma _{\sigma(v)}\,:\, v\in \sV \} $, the \emph{local automorphism} $(\sigma, \Phi )$ is the canonical automorphism of $\Gamma_\sG$ induced by the maps $\bigcup_{v\in \sV} \Gamma_v  \ni g \ra \phi_{\sigma(v)}(g)\in \Gamma_\sG$. One can easily observe that, under composition,  these form a subgroup of ${\rm Aut} (\Gamma_\sG)$ which is denoted by ${\rm Loc} (\Gamma_\sG)$. We denote by ${\rm Loc}_0 (\Gamma_\sG)$ the subgroup of local automorphisms satisfying $\sigma={\rm Id}$. Notice that ${\rm Loc}_0 (\Gamma_\sG)$ is naturally isomorphic to $\oplus_{v\in \sV} {\rm Aut }(\Gamma_v) $. Moreover, the inclusion  ${\rm Loc}_0 (\Gamma_\sG)\leqslant {\rm Loc} (\Gamma_\sG)$ has finite index.

\begin{thm}[\text{\cite{GM19}}]\label{outgp} Let $\Gamma_\sG$ be a graph product associated with a graph $\sG\in {\rm CC}_1$. Then its automorphism group ${\rm Aut}(\Gamma_\sG)$ is generated by the inner and the local automorphisms  of $\Gamma_\sG$. In fact we have ${\rm Aut}(\Gamma_\sG)= {\rm Inn} (\Gamma_\sG)\rtimes {\rm Loc(\Gamma_\sG)}$ and therefore \begin{equation}\label{autformulas}\begin{split}&
    {\rm Aut}(\Gamma_\sG)\cong \Gamma_\sG \rtimes  ( \left(\oplus_{v\in \sV} {\rm Aut }(\Gamma_v) \right ) \rtimes {\rm Sym }(\Gamma_\sG) );\\
    &
    {\rm Out}(\Gamma_\sG)\cong \left(\oplus_{v\in \sV} {\rm Aut }(\Gamma_v) \right ) \rtimes {\rm Sym }(\Gamma_\sG). 
    \end{split} 
\end{equation} 
Here, ${\rm Sym }(\Gamma_\sG)$ is an explicit finite subgroup of automorphisms of $\Gamma_\sG$. 
\end{thm}

\begin{proof} One can easily check that the graphs in CC$_1$ are atomic and therefore the conclusion follows immediately from \cite[Corollary C]{GM19}.\end{proof}

\begin{rem} a) If in the hypothesis of Theorem \ref{outgp} we assume in addition that $\{\Gamma_v\}_{v\in \sV} $ are pairwise non-isomorphic then we have ${\rm Sym }(\Gamma_\sG)=1$ in the automorphism group formulae \eqref{autformulas}. The same holds if instead we assume that any two cliques of $\sG$ have different cardinalities and for any $\sC\in {\rm cliq}(\sG)$ the set $\{\Gamma_v\}_{v\in \sC} $ consists of pairwise non-isomorphic subgroups.

b) One of the main goals of this paper is to establish both von Neumann algebraic and $C^*$-algebraic analogs of Theorem \ref{outgp}, under various assumptions on the vertex groups. For the specific statements in this direction, the reader may consult Corollaries \ref{symmetries4} and \ref{symmetries44} in Section \ref{symmetriesgpvna}.   
\end{rem} 


\section{Von Neumann algebraic cancellation in cyclic relations}\label{vNcyclicrel}

In this section we establish a von Neumann algebraic analog of Proposition \ref{combfacts} (part 3) describing the structure of all unitaries that satisfy a similar cyclic relation (Theorem \ref{cyclicrel2}). 
We start by first proving the following von Neumann algebraic counterpart of item 1) in Proposition \ref{combfacts}. 

\begin{lem}\label{cyclicrel2} Let $\La_1, \La_2, \Sigma <\Gamma$ be groups satisfying the following properties: \begin{enumerate} \item $\La_1 \cap 
\La_2=\La_1 \cap 
\Sigma=\La_2 \cap 
\Sigma=1$; and \item  for any $g_1\in \La_1\vee \Sigma$, $g_2\in \La_2\vee \Sigma$ satisfying $g_1g_2 \in \La_1 \vee \La_2$ one can find $a_1 \in \La_1$, $a_2 \in \La_2$ and $s\in \Sigma$ such that $g_1=a_1s$ and $g_2=s^{-1}a_2$.
\end{enumerate}

\noindent Then for any $y_1\in \mathscr U(\L(\La_1\vee \Sigma))$, $y_2\in \mathscr U(\L(\La_2\vee \Sigma))$ satisfying $y_1y_2 \in \mathscr U(\L(\La_1 \vee \La_2))$ one can find $v_1 \in \mathscr U( \L(\La_1))$, $v_2 \in \mathscr U(\L(\La_2))$ and $x\in \mathscr U(\L(\Sigma))$ so that $y_1=v_1 x$ and $y_2=x^{\ast}v_2$.

\end{lem}

Above we used the notation that if $\Gamma_1,\Gamma_2<\Gamma$ are  groups, then we denote by $\Gamma_1\vee\Gamma_2$ the subgroup of $\Gamma$ generated by $\Gamma_1$ and $\Gamma_2$.
 
\begin{proof}[Proof of Lemma \ref{cyclicrel2}] For each $i=1,2$ consider the Fourier expansion of $y_{i}= \sum_{g_i\in \La_i \vee \Sigma}(y_{i})_{g_i} u_{g_i} $. Since $y_{1} y_{2}\in \L(\La_1\vee \La_2)$ using condition 2.\ we have that 
\begin{equation*}\begin{split}y=y_{1} y_{2}&= \sum_{\substack{g_1\in  \La_1\vee \Sigma \\ g_2\in \La_2\vee \Sigma \\ g_1 g_2 \in \La_1\vee \La_2}} (y_1)_{g_1} (y_1)_{g_2} u_{g_1g_2}=\sum_{\substack{a_1\in \La_1\\ a_2\in \La_2\\ s\in \Sigma }}
(y_1)_{a_1 s } (y_2)_{s^{-1}a_2} u_{a_1a_2}\end{split}
\end{equation*}

\noindent The above formula and basic approximations show that
\begin{equation*} \begin{split}1=\sum_{s\in \Sigma } E_{\L(\La_1)}(y_1 u_{s^{-1}}) E_{\L(\La_2)}(u_s y_2) y^* 
\end{split}
\end{equation*}
where the right hand side quantity is only $\|\cdot \|_1$-summable.  Using this in combination with the Cauchy-Schwarz inequality we further get  
\begin{equation*}\begin{split}1&=\tau(\sum_{s\in \Sigma } E_{\L(\La_1)}(y_1 u_{s^{-1}}) E_{\L(\La_2)}(u_s y_2) y^* ) \leq \sum_{s\in \Sigma} |\tau(E_{\L(\La_1)}(y_1 u_{s^{-1}}) E_{\L(\La_2)}(u_s y_2) y^* )|\\
& \leq \sum_{s\in \Sigma } \|E_{\L(\La_1)}(y_1 u_{s^{-1}})\|_2 \|E_{\L(\La_2)}(u_s y_2) \|_2\\
&\leq \left(\sum_{s\in \Sigma } \|E_{\L(\La_1)}(y_1 u_{s^{-1}})\|^2_2\right)^{1/2} \left(\sum_{s\in \Sigma }\|E_{\L(\La_2)}(u_s y_2)\|^2_2\right)^{1/2}\leq \| y_1\|_2\|y_2\|_2=1.\end{split}
\end{equation*}

Thus, we must have equality in the Cauchy-Schwarz inequality, and hence, for every $s$ there is $c_s\in \mathbb C$ satisfying   \begin{equation}\label{equal1'}E_{\L(\La_1)}(y_1 u_{s^{-1}}) = c_s y E_{\L(\La_2)}(y_2^* u_{s^{-1}}).\end{equation}

Taking absolute values we get  $|E_{\L(\La_1)}(y_1 u_{s^{-1}})| = |c_s| | E_{\L(\La_2)}(y_2^* u_{s^{-1}})|$ and since  $\La_1\cap \La_2=1$ we conclude  $|E_{\L(\La_1)}(y_1 u_{s^{-1}})| = |c_s| | E_{\L(\La_2)}(y_2^* u_{s^{-1}})|\in 
\mathbb C 1$. Using the polar decomposition formula one can find $d_s, e_s\in \mathbb C $ and unitaries $x_s\in \L(\La_1), z_s \in \L(\La_2)$ satisfying \begin{equation*}\begin{split}&E_{\L(\La_1)}(y_1 u_{s^{-1}})=d_s x_s,\quad  E_{\L(\La_2)}(y_2^* u_{s^{-1}})=e_sz_s.\end{split}\end{equation*}   
Combining these with equation \eqref{equal1'} we get $d_s x_s =c_s e_s y  z_s$ for all $s\in \Sigma$; in particular, for every $d_s\neq 0$ we have {$ x_s = (e_s c_s/d_s) y z_s$}. Hence, for all $s,t\in \Sigma$ with $d_s,d_t\neq 0$ we have {$x_t^* x_s =(e_sc_s \overline{e_t c_t} / d_s \overline{d_t}) z_t^*z_s$.} Again, as  $\La_1\cap \La_2=1$ one can find $c_{s,t},d_{s,t}\in \mathbb C$  such that  \begin{equation}\label{equal2}x_s = c_{s,t} x_t\text{ and }z_s = d_{s,t} z_t.\end{equation} 

Fix $t\in \Sigma$. Using the prior relations we see that $y_1= \sum_{s\in \Sigma} E_{\L(\La_1)}(y_1 u_{s^{-1}})u_{s}=\sum_{s\in \Sigma} d_s x_s u_s= \sum_{s\in \Sigma} d_s c_{s,t}x_t u_s=x_t(\sum_{s\in \Sigma} d_s c_{s,t} u_s)$. In particular, this shows there are $v_{1}\in\mathscr{U}( \L(\La_1)),x \in \mathscr U(\L(\Sigma)) $ such that $y_1= v_1 x$. Similarly, the prior relations also imply that $y_2 = x^* v_2$ for some $v_2 \in \mathscr{U}(\L(\La_2))$.
\end{proof}

\begin{thm}\label{cyclerel1} Let $\sG$ be a  graph in the class ${\rm CC}_1$ and let $\sC_1,...,\sC_n$ be an enumeration of its consecutive cliques. Let $\Gamma_v$, $v\in \sV$ be a collection of icc groups and let $\Gamma_\sG$ be the corresponding graph product group.   For each $ i\in \overline{1,n}$ assume $x_{i,i+1}= a_{i,i+1} b_{i,i+1}$, where $a_{i,i+1}\in \sU(\L(\Gamma_{\sC_{i,i+1}}))$, $b_{i,i+1}\in \sU(\L(\Gamma_{\sC_i \cup \sC_{i+1}\setminus \sC_{i,i+1}})$. 

\noindent If $x_{1,2} x_{2,3}\cdots x_{n-1,n}x_{n,1}=1$, then for each $ i\in \overline{1,n}$ one can find $a_i \in \sU(\L(\Gamma_{\sC_{i-1,i}}))$, $b_i \in \sU(\L(\Gamma_{\sC^{\rm int}_i}))$, $c_i \in \sU(\L (\Gamma_{\sC_{i,i+1}}))$  so that $x_{i,i+1}= a_i b_i c_i b^*_{i+1} a^*_{i+2}c^*_{i+1}$. Here, we convene that $n+1=1$, $n+2=2$.     

\end{thm}

\begin{proof} Fix an arbitrary $i\in\overline{1,n}$. Using $x_{1,2} x_{2,3}\cdots x_{n-1,n}x_{n,1}=1$, it follows that $b_{i-1, i} b_{i,i+1}= a^*_{i-1,i} x_{i-2,i-1}^* \cdots x^*_{1,2} x^*_{n,1}\cdots x^*_{i+1,i+2}a^*_{i,i+1}$. { Since $a_{i-1,i},a_{i,i+1}\in \L(\Gamma_{\sG\setminus \sC^{\rm int}_{i}})$ and $x_{j,j+1},a_{j,j+1}\in \L(\Gamma_{\sC_j\cup \sC_{j+1}})$, for any $j\in\overline{1,n}$, we get that  $b_{i-1,i}  b_{i,i+1} = a_{i-1,i}^* x_{i-2,i-1}^* \cdots x^*_{1,2} x^*_{n,1}\cdots x^*_{i+1,i+2} a_{i,i+1}^*\in \L(\Gamma_{\sG\setminus \sC^{\rm int}_{i}})$.
Since $b_{i-1, i} b_{i,i+1}\in \L(\Gamma_{\sC_{i-1}\cup \sC_i \cup \sC_{i+1}})$  we deduce that} 
\begin{equation}\label{x5}
b_{i-1, i} b_{i,i+1}\in \sU(\L(\Gamma_{\sC_{i-1}\cup \sC_{i+1}})).    
\end{equation}

Now, fix two words $g_1\in \Gamma_{(\sC_{i-1}\cup \sC_i)\setminus \sC_{i-1,i}}, g_2\in \Gamma_{(\sC_{i}\cup \sC_{i+1})\setminus \sC_{i,i+1}}$ such that  $g_1g_2\in \Gamma_{
\sC_{i-1}\cup \sC_{i+1}}$. Using part 1.\ in Proposition \ref{combfacts} there exist $a_1\in \Gamma_{(\sC_{i-1}\setminus\sC_{i})\cup \sC_{i,i+1}}$, $b\in \Gamma_{(\sC_{i+1}\setminus \sC_{i})\cup \sC_{i-1,i}}$ and $s \in \Gamma_{\sC^{\rm int}_i}$ such that $g_1 = as$ and $g_2 = s^{-1}b$. Thus, applying Lemma \ref{cyclicrel2} for $\La_1 = \Gamma_{(\sC_{i-1}\setminus\sC_{i})\cup \sC_{i,i+1}}
$, $\La_2 = \Gamma_{(\sC_{i+1}\setminus \sC_{i})\cup \sC_{i-1,i}}$ and $\Sigma =\Gamma_{\sC^{\rm int}_i}$, we derive from \eqref{x5} that one can find  unitaries $x_{i-1}\in \L(\Gamma_{(\sC_{i-1}\setminus\sC_{i})\cup \sC_{i,i+1}}),z_{i-1} \in \L(\Gamma_{\sC^{\rm int}_i}) $ such that $b_{i-1,i}= x_{i-1} z_{i-1}$.

Lemma \ref{cyclicrel2} also implies from \eqref{x5} that $b_{i,i+1}= z^*_{i-1}y_i$ for some $y_i\in \sU(\L(\Gamma_{(\sC_{i+1}\setminus\sC_{i})\cup \sC_{i-1,i}}) )$. Using this with  $b_{i,i+1}= x_{i} z_{i}$, we get that $b_{i,i+1}= z^*_{i-1}y_i= x_i z_i$. Hence, $x_i^* z^*_{i-1}= z_iy_i^*=:t_{i,i+1}$ and note that $t_{i,i+1}\in \L( \Gamma_{\sC_{i-1,i} \cup \sC_{i+1,i+2}})$.
Thus, $x_{i,i+1}= a_{i,i+1}b_{i,i+1}= a_{i,i+1}z^*_{i-1} t_{i,i+1}^* z_i= z^*_{i-1} a_{i,i+1}t_{i,i+1}^* z_i$. In conclusion, we showed  that for any $i\in\overline{1,n}$ we have 
\begin{equation}\label{pp5} x_{i,i+1}=z^*_{i-1} a_{i,i+1}t_{i,i+1} z_i.\end{equation}

Now, we note that since $x_{1,2} x_{2,3}\cdots x_{n-1,n}x_{n,1}=1$, then we obviously have $a_{1,2}t_{1,2} a_{2,3} t_{2,3}\cdots a_{n-1,n} t_{n-1,n}a_{n,1}t_{n,1}=1$. Again we will use this relation together with the same argument from the proof of Lemma \ref{cyclicrel2} to show that $x_{i,i+1}$ has the form described in the conclusion of the theorem. First, observe the cyclic relation and a similar argument as in the beginning of the proof show that $ w:=t_{i-1,i}  a_{i,i+1}t_{i,i+1} a_{i+1,i+2} t_{i+1,i+2}a_{i+2,i+3}= a^*_{i-1,i}\cdots t^*_{1,2}a^*_{1,2} t^*_{n,1}a^*_{n,1}\cdots t^*_{i+2,i+3} \in \L(\Gamma_{\sC_{i-2,i-1} \cup \sC_{i-1,i}  \cup \sC_{i+1,i+2}  \cup \sC_{i+2,i+3} }) $.

Now, fix three words $w_1\in \Gamma_{\sC_{i-2,i-1} \cup \sC_{i,i+1} }, w_2\in \Gamma_{\sC_{i-1,i} \cup \sC_{i+1,i+2} }, w_3\in \Gamma_{\sC_{i,i+1} \cup \sC_{i+2,i+3} }$ satisfying $w_1w_2 w_3\in \Gamma_{\sC_{i-2,i-1} \cup \sC_{i-1,i}  \cup \sC_{i+1,i+2}  \cup \sC_{i+2,i+3} }$. Using part 2.\ of Proposition \ref{combfacts} we must have that $w_1= as$  and $w_3 = s^{-1}b$ where $a\in \Gamma_{\sC_{i-2,i-1} }$, $b\in \Gamma_{\sC_{i+2,i+3} }$ and $s \in \Gamma_{\sC_{i,i+1} }$. Since
$t_{i,i+1} a_{i+1,i+2}\in \L( \Gamma_{\sC_{i-1,i} \cup \sC_{i+1,i+2}})$, we can write the Fourier expansions of $t_{i-1,i} a_{i,i+1}= \sum_{w_1\in \sC_{i-2,i-1} \cup \sC_{i,i+1}} (t_{i-1,i} a_{i,i+1} )_{w_1}u_{w_1}$ and  $t_{i+1,i+2} a_{i+2,i+3}= \sum_{w_3\in \sC_{i,i+1} \cup \sC_{i+2,i+3}} (t_{i+1,i+2} a_{i+2,i+3} )_{w_3}u_{w_3}$. All these observations imply that 
\begin{equation*}\begin{split}&w=(t_{i-1,i}  a_{i,i+1})(t_{i,i+1} a_{i+1,i+2}) (t_{i+1,i+2}a_{i+2,i+3})=\\
&=\sum_{s \in \Gamma_{\sC_{i,i+1} }} E_{\L(\Gamma_{\sC_{i-2,i-1} })}(t_{i-1,i}  a_{i,i+1} u_{s^{-1}})t_{i,i+1} a_{i+1,i+2} E_{\L(\Gamma_{ \sC_{i+2,i+3} })}(u_s t_{i+1,i+2}a_{i+2,i+3}), 
\end{split}
\end{equation*}
where again the convergence is in $\|\cdot \|_1$.
Using this and Cauchy-Schwarz inequality we get
 \begin{equation*}\begin{split}1&=\sum_{s \in \Gamma_{\sC_{i,i+1} }}| \tau(E_{\L(\Gamma_{\sC_{i-2,i-1} })}(t_{i-1,i}  a_{i,i+1} u_{s^{-1}})t_{i,i+1} a_{i+1,i+2} E_{\L(\Gamma_{ \sC_{i+2,i+3} })}(u_s t_{i+1,i+2}a_{i+2,i+3}) 
w^*)|\\
& \leq \sum_{s \in \Gamma_{\sC_{i,i+1} }} \|E_{\L(\Gamma_{\sC_{i-2,i-1} })}(t_{i-1,i}  a_{i,i+1} u_{s^{-1}})\|_2 \|E_{\L(\Gamma_{ \sC_{i+2,i+3} })}( a_{i+2,i+3}^*t_{i+1,i+2}^*u_{s^{-1}})\|_2\\
& \leq \left(\sum_{s \in \Gamma_{\sC_{i,i+1} }} \|E_{\L(\Gamma_{\sC_{i-2,i-1} })}(t_{i-1,i}  a_{i,i+1} u_{s^{-1}})\|^2_2\right)^{1/2} \cdot \\
&\cdot \left(\sum_{s \in \Gamma_{\sC_{i,i+1} }}\|E_{\L(\Gamma_{ \sC_{i+2,i+3} })}( a_{i+2,i+3}^*t_{i+1,i+2}^*u_{s^{-1}})\|^2_2\right)^{1/2} \\
&\leq \| t_{i-1,i}  a_{i,i+1}\|_2\|a_{i+2,i+3}^*t_{i+1,i+2}^*\|_2=1.\end{split}
\end{equation*}

Therefore, one can get scalars $c_s$ such that  \begin{equation}\label{equal1}E_{\L(\Gamma_{\sC_{i-2,i-1} })}(t_{i-1,i}  a_{i,i+1} u_{s^{-1}})= c_s w E_{\L(\Gamma_{ \sC_{i+2,i+3} })}( a_{i+2,i+3}^*t_{i+1,i+2}^*u_{s^{-1}})a_{i+1,i+2}^* t_{i,i+1}^*.\end{equation} 
Thus, proceeding in the same fashion as in the proof of Lemma \ref{cyclicrel2} one can find $d_s,e_s\in \mathbb C $, $g_s\in \sU(\L(\Gamma_{\sC_{i-2,i-1}})) $, $h_s\in \sU(\L(\Gamma_{\sC_{i+2,i+3}})) $  so that 
\begin{equation}\label{pp1}
\begin{split}&E_{\L(\Gamma_{\sC_{i-2,i-1} })}(t_{i-1,i}  a_{i,i+1} u_{s^{-1}})=d_s g_s,\\
&E_{\L(\Gamma_{ \sC_{i+2,i+3} })}( a_{i+2,i+3}^*t_{i+1,i+2}^*u_{s^{-1}})a_{i+1,i+2}^* t_{i,i+1}^*=e_sh_s.\end{split}
\end{equation}   

Hence,  \eqref{equal1} gives that $d_s g_s= c_s e_s w h_s$ for all $s\in \Gamma_{\sC_{i,i+1}}$ and finally employing the same arguments as the first part one can find scalars $c_{s,t}, d_{s,t}$ such that $g_s = c_{s,t} g_t$ and $h_s= d_{s,t} h_t$ for all $s,t\in \Gamma_{\sC_{i,i+1}}$. Using \eqref{pp1}, we derive that 
\begin{equation*}
 t_{i-1,i}a_{i,i+1}=\sum_{s\in \Gamma_{\sC_{i,i+1}}} d_sg_su_s= g_e \sum_{s\in \Gamma_{\sC_{i,i+1}}} d_s c_{s,e} u_s.
\end{equation*}
This further implies that one can find unitaries $r_{i-1} \in \sU(\L(\Gamma_{\sC_{i-2,i-1}}))$,  $p_{i-1}\in \sU(\L(\Gamma_{\sC_{i,i+1}}))$ such that $t_{i-1,i}a_{i,i+1}= r_{i-1} p_{i-1}$, and hence, $t_{i-1,i}= r_{i-1} p_{i-1} a_{i,i+1}^*$.  Similarly, we get  $t_{i,i+1}= r_{i}p_{i} a_{i+1,i+2}^*$, and hence, from \eqref{pp5} we deduce that 
\begin{equation}\label{pp6}
    x_{i,i+1}= z^*_{i-1}a_{i,i+1}t_{i,i+1} z_i = z^*_{i-1}a_{i,i+1} r_{i}p_{i} a_{i+1,i+2}^* z_i= r_i z^*_{i-1} a_{i,i+1}   z_i p_i a_{i+1,i+2}^*. 
\end{equation} 
Now, one can see that using the cyclic relation $x_{1,2}\dots x_{n-1,n}x_{n,1}=1$, we get that $p_i = r^*_{i+2 }$. This together with \eqref{pp6} give the desired conclusion by taking $a_i=r_i$, $b_i=z^*_{i-1}$ and $c_i=a_{i,i+1}$.
\end{proof}

\section{Rigid subalgebras of graph product groups von Neumann algebras}
In this section we classify all rigid subalgebras of von Neumann algebras associated with graph product groups. This should be viewed as a counterpart of \cite[Theorem 4.3]{IPP05} for amalgamated free products von Neumann algebras. In fact, the later plays an essential role in deriving our result. For convenience, we include a detailed proof on how it follows from this.  

\begin{thm}\label{proptcliques}\label{theorem.embed.clique} Let $\Gamma= \sG\{\Gamma_v\}$ be a graph product group, let $\Gamma\ca \P$ be any trace preserving action and denote by $\M=\P\rtimes \Gamma$ the corresponding crossed product von Neumann algebra. Let $r\in \M$ be a projection and let $\Q\subset r\M r$ be a property (T) von Neumann subalgebra. 

Then one can find a clique $\sC\in {\rm cliq}(\sG)$ such that $\Q\prec_\M \P\rtimes \Gamma_\sC$. Moreover, if $Q\nprec \P \rtimes \Gamma_{\sC \setminus \{c\}}$ for all $c\in \sC$, then one can find projections $q\in \Q$, $q'\in \Q'\cap r\M r$ with $qq'\neq 0$ and a unitary $u\in \M$ such that $u q\Q q q'u^{*}\subseteq \P \rtimes \Gamma_\sC$.   In particular, if $\P\rtimes \Gamma_\sC$ is a factor then one can take $q=1$ above.  

\end{thm}

\begin{proof}
Let $\mathscr G_0=(\mathscr V_0,\mathscr E_0) \subseteq \mathscr G=(\mathscr V, \mathscr E)$ be a subgraph with $|\mathscr V_0|$ minimal such that $\Q \prec \P \rtimes \Gamma_{\mathscr G_0}$. In the remaining part we show that $\mathscr G_0$ is complete, which proves the conclusion.

Denote   $\N=\P \rtimes \Gamma_{\sG_0}$. Since $\Q \prec_\M  \N$ one can find projections $q\in \Q$ $p\in \N$, a non-zero partial isometry $v\in p\M q$ and a $\ast$-isomorphism onto its image $\theta: q\Q q\ra  \R:= \theta(q\Q q)\subseteq p \N p$ such that $\theta(x)v=vx$ for all $x\in q \Q q$. Notice that $vv^*\in \R'\cap p\M p$ and $v^*v\in (\Q'\cap \M)q$. Moreover, one can assume without any loss of generality that the support projection of $E_\N(vv^*)$ equals $p$.   

 Assume by contradiction that $\sG_0$ is not complete. Thus, one can find $v\in \sG_0$ so that $\Ga_{\sG_0}$ admits a non-canonical amalgam decomposition $\Gamma_{\sG_0}=\Gamma_{\sG_0\setminus \{v\}}\ast_{\Gamma_{\text{lk}(v)}} \Gamma_{\text {st}(v)}$; in particular, we have that $|\text{st}(v)|\leq|\sV_0|-1$. Since $\Q$ has property (T), $\R$ has property (T) as well. Using \cite[Theorem 5.1]{IPP05} we have either i)
 $\R  \prec_\N  \P \rtimes \Gamma_{\sG_0\setminus \{v\}}$ or ii) $ \R\prec_\N  \P\rtimes \Gamma_{\text{st}(v)}$.
Assume i). Denote by  $\X:=\P \rtimes \Gamma_{\sG_0\setminus \{v\}}$.  
As $\R \prec_\N  \X$ one can find projections $e\in \R$ $f\in \X$, a non-zero partial isometry $w\in f\N e$ and a $\ast$-isomorphism onto its image $\psi: e\R e\ra  \T:= \theta(e\R e)\subseteq f \X f$ such that $\psi(x)w=wx$ for all $x\in e \R e$. 

Next, we argue that $wv\neq 0$. Otherwise, we would have $0= wvv^*$ and since $w\in \N$ we get $0= w E_\N(vv^*)$. Therefore $0= w s(E_\N(vv^*))=w p=w$, which is a contradiction. Combining the previous intertwining relations, we get $\phi(\theta(x))wv =w \theta(x)v =wv x$, for all $x\in t\Q t $; here we denoted by $0\neq t= \theta^{-1}(e)$. Taking the polar decomposition of $wv$ in the prior intertwining relation, we obtain that $\Q \prec_\M \X= \P \rtimes \Gamma_{\sG_0\setminus \{v\}}$. However, since $|\mathscr V_0\setminus \{v\}|=|\mathscr V_0|-1$, this contradicts the minimality of $|\mathscr V_0|$. In a similar manner, one can show case ii) also leads to a contradiction.

Next, we show the moreover part.
Let $\S = \P \rtimes \Gamma_{\sC}$. From the first part one can find projections $q\in \Q$ $s\in \S$, a non-zero partial isometry $v_0\in s\M q$ and a $\ast$-isomorphism onto its image $\theta: q\Q q\ra  \Y:= \theta(q\Q q)\subseteq s \S s$ such that $\theta(x)v_0=v_0x$ for all $x\in q \Q q$. Notice that $v_0v^*_0\in \Y'\cap s\M s$ and $v^*_0v_0\in q\Q q'\cap q\M q$. Moreover, one can assume without any loss of generality that the support projection of $E_\S(v_0v^*_0)$ equals $s$. 
Observe we have an amalgamated free product  decomposition $\M = (\P\rtimes \Gamma_{\sV\setminus \{c\}}) \ast_{\P \rtimes \Gamma_{\sC\setminus \{c\}}} (\P \rtimes \Gamma_\sC)$. Using the same argument as before, since $\Q\nprec_{\M} \P \rtimes \Gamma_{\sC \setminus \{c\}}$, we must have that $\Y\nprec_{\S} \P \rtimes \Gamma_{\sC\setminus \{c\}} $. Therefore  by \cite[Theorem 1.2.1]{IPP05} we have that $v_0v_0^*\in \S$ and hence the intertwining relation implies that $v_0 q \Q q v_0^*= \Y v_0v_0^*\subseteq \S$. If $u$ is a unitary extending $v_0$ we further see that {$u q\Q q v_0^*v_0 u^*\subseteq \S$}. Letting $q'=v_0^*v_0$ we get the desired conclusion.

To see the last part just notice that, since $\P\rtimes \Gamma_\sC$ is a factor, after passing to a new unitary $u$, one can replace above $q$ with its central support in $\Q$.\end{proof}

\section{Symmetries of graph product group von Neumann algebras}\label{symmetriesgpvna}

The main result of this section is a strong rigidity result describing all $\ast$-isomorphisms between factors associated with a fairly large family of graph product groups arising from finite graphs in the class CC$_1$ (Theorem \ref{symmetries3}). As a by-product we obtain concrete descriptions of all symmetries of these factors including such examples with trivial fundamental groups (Corollaries \ref{symmetries4} and \ref{fg}). However, to be able to state and prove these results we first need to introduce some new terminology and establish a few preliminary results.

\subsection{Local isomorphisms of graph product von Neumann algebras.}\label{locisomsec} The isomorphism class of a  von Neumann algebra associated with a graph product group tends to be fairly abundant. As in the group situation, a rich source of isomorphisms  stems from both the isomorphism class of the underlying graph and the isomorphism classes of the von Neumann algebras of the vertex groups.  By analogy with the group case, these are called local isomorphisms and we briefly explain their construction below.  

Let $\sG$, $\sH$ be simple finite graphs and let $\Gamma_ \sG$ and $\La_ \sH$ be graph products groups, where their vertex groups are $\{\Gamma_v, v\in \sV\}$ and $\{\La_w, w\in \sW \}$, respectively. Assume $\sG$ and $\sH$ are isometric and fix $\sigma :\sG \ra\sH$ an isometry. In addition, assume that $\Phi=\{\Phi^\sigma_v, v\in \sV\}$ is a collection of $\ast$-isomorphisms  $\Phi^\sigma_v: \L(\Gamma_v) \ra \L(\La_{\sigma(v)})$ for all $v\in \sV$. Then the following holds. 

\begin{thm}\label{localisom} There exists a unique $\ast$-isomorphism denoted by  $(\Phi, \sigma):\L(\Gamma_\sG) \ra \L(\La_\sH)$ which extends the maps $\bigcup_{v\in \sV} \L (\Gamma_v) \ni x \ra \Phi^\sigma_v(x)\in \L(\La_\sH)$.   

\end{thm}
\begin{proof}  We recall from \cite[Definition 1.2]{CF14} that a word for $\sG$ is a finite sequence ${\bf v}=(v_1,\dots,v_n)$ of elements in $\sV$. The word ${\bf v}$ is called reduced if whenever $i<j$ and $v_{i+1},\dots,v_{j-1}\in {\rm st}(v_j)$, then $v_i\neq v_j$. Following \cite[Section 2.3]{CF14}, $\L (\Gamma_\sG)$ can be presented alternatively as the graph product von Neumann algebra associated to the graph $\sG$ and vertex von Neumann algebras $\{\L(\Gamma_v)\}_{v\in \sV}$. 

We continue by proving the following claim: for any reduced word $(v_1,\dots,v_n)$ in $\sG$ and elements $a_i\in \L(\Gamma_{v_i})$ with $\tau(a_i)=0$, we have $\tau (\Phi^\sigma_{v_1}(a_1) \cdots \Phi^\sigma_{v_n}(a_n))=0$. For showing this, denote $w_i=\sigma (v_i)\in \sW$ and $b_i=\Phi^\sigma_{v_i}(a_i) \in \L(\Lambda_{w_i})$ for any $i$.
Note that the word $(w_1,\dots,w_n)$ is reduced in $\sH$ and $\tau(b_i)=0$, for any $i$. By considering the Fourier series of $b_i$, the claim follows by proving that whenever $h_i\in \Lambda_{w_i}$ with $h_i\neq 1$, then $h_1\cdots h_n\neq 1$. Since $(w_1,\dots,w_n)$ is a reduced word in $\sH$, it is easy to see that $h_1\cdots h_n$ is a reduced element of $\Lambda_{\sH}$ in the sense of \cite[Definition 3.5]{Gr90}. By applying \cite[Theorem 3.9]{Gr90}, it implies that $h_1\cdots h_n\neq 1$, hence proving the claim. 

Finally, our theorem follows now directly by applying \cite[Proposition 2.22]{CF14}.
\end{proof}

Throughout this section, $(\Phi,\sigma)$ will be called the \emph{local isomorphism} induced by $\sigma$ and $\Phi=\{\Phi^\sigma_v, v\in \sV\}$. When $\sG=\sH$ and $\Gamma_v=\La_v$ for all $v$, these are called \emph{local automorphisms} and they form a subgroup of ${\rm Aut} (\L(\Gamma_\sG))$ under composition which will be denoted by ${\rm Loc}_{\rm v,g}(\L(\Gamma_\sG))$. The subgroup of local automorphism satisfying $\sigma={\rm Id}$ is denoted by  ${\rm Loc}_{\rm v}(\L(\Gamma_\sG))$ and observe that it has finite index in  $ {\rm Loc}_{\rm v}(\L(\Gamma_\sG))$. Moreover, we have  ${\rm Loc}_{\rm v}(\L(\Gamma_\sG))= \oplus_{v\in \sV}  {\rm Aut}(\L(\Gamma_v))$.
Next, we observe that most of the time $(\Phi,\sigma)$ is an outer automorphism.

\begin{prop}\label{outerla} Under the same assumptions as before, suppose in addition that $\sG$ is a graph satisfying $\cap_{v\in \sV} {\rm star}(v)=\emptyset$. Then $(\Phi,\sigma)$ is inner if and only if $\sigma ={\rm Id}$ and $\Phi^\sigma_v={\rm Id}$ for all $v\in \sV$. 

\end{prop}
\begin{proof} Let $\M =\L(\Gamma_\sG)$ and let $u\in \sU(\M)$ such that  $(\Phi,\sigma)= {\rm ad}(u)$. Fix $v\in \sV$. From definitions we have $u \L(\Gamma_v) u^*= \L(\Gamma_{\sigma(v)})$. Using Theorem \ref{controlquasinormalizer1} we get that $v=\sigma (v)$ and $u\in \sU(\L(\Gamma_{{\rm star}(v)}))$. As this holds for all $v\in \sV$ we get $\sigma={\rm Id}$ and also $u \in \bigcap_{v\in \sV} \L(\Gamma_{{\rm star}(v)})= \mathbb C 1$. Hence $(\Phi,\sigma)={\rm Id}$ and also $\Phi^{\sigma}_v={\rm Id}$  for all $v\in \sV$.   \end{proof}

When $\sigma={\rm Id}$, let ${\rm Loc}_{\rm v, i}(\L(\Gamma_\sG))$ be the set of all local automorphisms $(\Phi,\sigma)$ which satisfies that for any $v\in \mathscr V$, there exists a  unitary  $u_v \in \L(\Gamma_v)$ such that $\Phi^\sigma_v={\rm ad}(u_v)$.  It is easy to see that it forms a normal subgroup of ${\rm Loc}_{\rm v}(\L(\Gamma_\sG))$ under composition. Thus, when there exists an $v\in\mathscr V$ for which $\Gamma_v$ is an icc group, it follows from Proposition \ref{outerla}  that ${\rm Loc}_{\rm v, i} (\L(\Gamma_\sG))$,
and hence ${\rm Out} (\L(\Gamma_\sG))$, is always an uncountable group. In conclusion, for this class of von Neumann algebras, in general, one cannot expect rigidity results and computations of their symmetries of the same precision level with the prior results \cite{PV06,Va08}.

\begin{rem}
It is worth mentioning that the class of local isomorphisms can be defined for all tracial graph products \cite{CF14} (regardless if they come form groups or not) with essentially the same proofs.    
\end{rem}  

Next, we highlight a family of  $\ast$-isomorphisms between graph product von Neumann algebras that is specific to graphs in class ${\rm CC}_1$ and is related more to the clique algebras structure rather than the vertex algebra structure as in the previous part.  As before, let $\sG,\sH \in {\rm CC}_1$ be isomorphic graphs and fix $\sigma : \sG \ra \sH$ an isometry. Let ${\rm cliq}( \sG) =\{ \sC_1, \ldots ,\sC_n\}$ be a consecutive enumeration of the cliques of $\sG$. Let $\Gamma_\sG$ and $\Lambda_\sH$ be graph product groups and assume for every $i\in\overline{1,n}$ there are $\ast$-isomorphisms $\theta_{i-1,i}: \L(\Gamma_{\sC_{i-1,i}})\ra \L(\Lambda_{\sC_{\sigma (\sC_{i-1,i})}}) $ and   $\xi_{i}:\L(\Gamma_{\sC^{\rm int}_i})\ra \L(\Lambda_{\sigma(\sC^{\rm int}_i)})$. Here, and afterwards, we use the notation $\sC_{0,1}=\sC_{n,1}$. Using Lemma \ref{altgpdecomp} we can view $\Gamma_\sG$ as a graph product group $\Gamma'_{\mathcal T_n}$ over the graph $\mathcal T_n$ where the vertex groups satisfy $\Gamma'_{w_i} = \oplus_{v\in \sC^{\rm int}_i}\Gamma_v$ and $\Gamma'_{b_i}=\oplus_{v\in \sC_{i-1,i}} \Gamma_v$. Similarly, $\La_\sH= \La'_{\mathcal T_n}$  where $\La'_{\sigma(w_i)} = \oplus_{v\in \sC^{\rm int}_i}\La_{\sigma(v)}$ and $\La'_{\sigma(b_i)}=\oplus_{v\in \sC_{i-1,i}} \La_{\sigma(v)}$. Therefore, using Theorem \ref{localisom} these isomorphisms induce a unique $\ast$-isomorphism $\phi_{\theta,\xi, \sigma}:\L(\Gamma_\sG)\ra  \L(\Lambda_\sH)$ defined as follows

\begin{equation}\label{branchaut} \phi_{\theta,\xi, \sigma}(x)=\begin{cases}
\theta_{i-1,i}(x),  \text{ if }  x\in \L(\Gamma_{\sC_{i-1,i}})\\
\xi_i(x), \text{ if }  x\in \L(\Gamma_{\sC^{\rm int}_i }),
\end{cases} \qquad\qquad\qquad\qquad
\end{equation}
for all $i\in\overline{1,n}$. 

When $\Gamma_\sG =\La_\sH$ this construction yields a group of  automorphisms of $\L(\Gamma_\sG)$ that will be denoted by ${\rm Loc}_{\rm c,g}(\L(\Gamma_\sG))$. We also denote by ${\rm Loc}_{\rm c}(\L(\Gamma_\sG))$ the subgroup of all automorphisms satisfying $\sigma={\rm Id}$. Notice that ${\rm Loc}_{\rm c}(\L(\Gamma_\sG))\cong \oplus_i {\rm Aut}(\L(\Gamma_{\sC_{i-1,i}}))\oplus {\rm Aut}(\L(\Gamma_{\sC^{\rm int}_{i}}))$ also ${\rm Loc}_{\rm c}(\L(\Gamma_\sG))\leqslant {\rm Loc}_{\rm c,g}(\L(\Gamma_\sG))$ has finite index. 

Next, we highlight a subgroup of automorphisms in ${\rm Loc}_{\rm c}(\L(\Gamma_\sG))$ that will be useful in stating our main results. Namely, consider a family of nontrivial unitaries  $a_{i-1,i} \in \L(\Gamma_{\sC_{i-1,i}})$ and $b_i\in \L(\Gamma_{\sC^{\rm int}_i})$ for every $i\in\overline{1,n}$. If in the formula \eqref{branchaut} we let $\theta_{i-1,i}= {\rm ad} (a_{i-1,i}) $ and $\xi_i ={\rm ad} (b_i)$, then the corresponding  automorphism $\phi_{\theta,\xi, {\rm Id}}$ is an (most of the times outer) automorphism of $\L(\Gamma)$ which we will denote by $\phi_{a,b}$ throughout this section. The set all all such automorphisms form a normal subgroup denoted by ${\rm Loc}_{\rm c, i}(\L(\Gamma_\sG))\lhd{\rm Loc}_{\rm c}(\L(\Gamma_\sG))$.  From definitions we also have that ${\rm Loc}_{\rm v, i}(\L(\Gamma_\sG)) < {\rm Loc}_{\rm c, i}(\L(\Gamma_\sG))$ and ${\rm Loc}_{\rm v}(\L(\Gamma_\sG)) < {\rm Loc}_{\rm c}(\L(\Gamma_\sG))$.

\begin{prop} If  $\phi_{a,b}\in {\rm Loc}_{\rm c, i}(\L(\Gamma_\sG))$ is inner if and only if   $a_{i,i+1}\in \sZ (\L(\Gamma_{\sC_{i-1,i}}))$ and $b_i\in \sZ (\L(\Gamma_{\sC^{\rm int}_i}))$, for all $i\in\overline{1,n}$.

\end{prop}

\begin{proof}
Assume first that  $\phi_{a,b}\in {\rm Loc}_{\rm c, i}(\L(\Gamma_\sG))$ is inner, and hence, there is a unitary $u\in \L(\Gamma_\sG)$ such that $\phi_{a,b}(x)u=ux$, for any $x\in\L(\Gamma_\sG)$. Fix an arbitrary $i\in\overline{1,n}$. Then for any  $x\in \L(\Gamma_{\sC_{i-1,i}})$,  we have that $a_{i-1,i}xa_{i,i-1}^*=uxu^*$. Using this together with Theorem \ref{controlquasinormalizer1} to derive that $u^*a_{i,i-1}\in \L(\Gamma_{\sC_{i-1,i}})'\cap \L(\Gamma_\sG)\subset \L(\Gamma_{\sC_{i-1}\cup\sC_i})$. Since $a_{i,i+1}\in  \L(\Gamma_{\sC_{i-1,i}})$, it follows that $u\in \bigcap_{i=1}^n \L(\Gamma_{\sC_{i-1}\cup\sC_i})=\mathbb C 1$, it follows that $a_{i,i+1}\in \sZ (\L(\Gamma_{\sC_{i-1,i}}))$. Similarly, one can show that $b_i\in \sZ (\L(\Gamma_{\sC^{\rm int}_i}))$, for all $i\in\overline{1,n}$. This concludes one direction of the proof. As for the converse, note that we trivially have $\phi_{a,b}={\rm Id}$.
\end{proof}

\subsection{Computations of symmetries of graph product von Neumann algebras.} Next, we introduce a few preliminary results needed to describe the isomorphisms between von Neumann algebras arising from graph products with property (T) vertex groups.

\begin{thm}\label{symmetries1} Let $\Gamma= \sG\{\Gamma_v\}$ and $\Lambda=\sH\{\Lambda_w\}$ be graph product groups and assume that $\Gamma_v$ and $\Lambda_w$ are icc property (T) groups for all $v\in \sV$ and $w \in \sW$. Let  $\theta: \L(\Gamma) \ra \L(\Lambda)$ be any $\ast$-isomorphism. Then there is a bijection $\sigma: {\rm cliq}(\sG)\ra {\rm cliq}(\sH)$ such that for every $\sC\in {\rm cliq}(\sG)$ there is a unitary $u_\sC\in \L(\Lambda)$ such that $\theta (\L(\Gamma_\sC)) = u_\sC \L(\Lambda_{\sigma(\sC)})u_\sC^*$.   
\end{thm}

\begin{proof} Fix $\sC\in {\rm cliq}(\sG)$. Using the  hypothesis and Corollary \ref{controlquasinormalizer2} it follows that  {$\theta(\L(\Gamma_\sC))\subseteq \L(\Lambda)$} is a property (T) irreducible subfactor. By the first part of Theorem \ref{proptcliques} there exists a clique $\sigma(\sC)\in {\rm cliq}(\sH)$ such that $\theta(\L(\Gamma_\sC))\prec_{\L(\Lambda)} \L(\Gamma_{\sigma(\sC)})$. Now, we argue that for every $c\in \sigma(\sC)$ we have that  $\theta(\L(\Gamma_\sC))\nprec_{\L(\Lambda)} \L(\Gamma_{\sigma(\sC)\setminus \{c\}})$. Indeed, if we assume  that $\theta(\L(\Lambda_\sC))\prec_{\L(\Lambda)} \L(\Gamma_{\sigma(\sC)\setminus \{c\}})$, then by passing to relative commutants intertwining we would get from  \cite[Lemma 3.5]{Va08} that $\L(\Lambda_c)=\L(\Lambda_{\sigma(\sC)\setminus \{c\}})'\cap \L(\Lambda) \prec_{\L(\Lambda)} \theta(\L(\Gamma_\sC))'\cap \L(\Lambda)= \theta( \L(\Gamma_\sC)'\cap \L(\Gamma))= \mathbb C 1$, which is a contradiction.
Thus, by using that $\L(\Lambda_{\sigma(\sC)})$ is a factor and
$\L(\Gamma_\sC)\subset \L(\Gamma)$ irreducible, it follows from the moreover part of Theorem \ref{proptcliques} that there is a unitary $u_\sC \in \L(\Lambda)$ satisfying $u_\sC\theta(\L(\Gamma_\sC))u^*_\sC\subseteq \L(\Lambda_{\sigma(\sC)})$.

Reversing the roles of $\Gamma$ and $\Lambda$, in a similar manner for every $\sD \in {\rm cliq} (\sH)$ one can find $\tau(\sD)\in{\rm cliq}(\sG)$ and a unitary $w_{\sD}\in \L(\Gamma)$ satisfying $\L(\Lambda_{\sD}) \subseteq w_\sC\theta(\L(\Gamma_{\tau(\sD)})w^*_\sD. $ Altogether these show that $u_\sC\theta(\L(\Gamma_\sC))u^*_\sC\subseteq \L(\Lambda_{\sigma(\sC)})\subseteq w_{\sigma(\sC)}\theta(\L(\Gamma_{\tau(\sigma(\sC))})w^*_{\sigma(\sC)}$. In particular, Theorem \ref{controlquasinormalizer1} implies that  $\sC =\tau(\sigma(\sC))$ and $u_\sC^* w_{\sigma(\sC)}\in \theta (\L(\Gamma_\sC))$. This combined with the prior containment imply that $u_\sC\theta(\L(\Gamma_\sC))u^*_\sC= \L(\Lambda_{\sigma(\sC)})$. As $\sC= \tau(\sigma(\sC))$ for any clique $\sC$ of $\sG$, it follows in particular that $\sigma$ is a bijection.      
\end{proof}

\noindent {\bf Remarks.} The above theorem still holds under the more general assumption that each vertex group  $\Gamma_v$ possesses an infinite property (T) normal subgroup. The proof is essentially the same and it is left to the reader.

We continue by recording a notion of unique prime factorization along with some examples that will be needed in the first main result.

\begin{defn}\label{supf}
A family $\mathcal C$ of countable icc groups is said to satisfy the  \emph{s-unique prime factorization} if whenever $\M=\L(\Gamma_1\times\dots\times\Gamma_m)^t=\L(\Lambda_1\times\dots\times\Lambda_n)$ for some $\Gamma_1,\dots,\Gamma_m$, $\Lambda_1,\dots,\Lambda_n$ that belong to $\mathcal C$ and $t>0$, we must have that
 $t=1$, $m=n$ and there exist a unitary $u\in \M$ and a permutation $\tau\in \mathfrak S_n$ such that $u \L(\Gamma_i)u^*= \L(\Lambda_{\tau(i)})$, for all $i\in\overline {1,n}$.
\end{defn}

There are several classes of natural examples of groups that satisfy this unique factorization condition in the literature, but for our paper will be relevant the ones which have property (T). Thus appealing to  the results in \cite{CDK19,CDHK20,Da20,CIOS21},  we have the following: 

\begin{cor}\label{corollary.supf} Class  $\mathcal C$ satisfies the s-unique prime factorization whenever $\mathcal C$ is one of the following: 
\begin{enumerate}[(1)]
    \item the class of property (T) fibered Rips construction as in \cite{CDK19,CDHK20};    
    \item the class of property (T) generalized wreath-like product groups $\mathcal W\mathcal R (A, B \ca I)$ where $A$ is abelian, $B$ is icc subgroup of a hyperbolic group and the action $B\ca I$ has infinite stabilizers \cite{CIOS21}. 
\end{enumerate}

\end{cor}

\begin{proof}
    Part (1) is a direct consequence of \cite{CDK19,CDHK20,Da20}. Part (2) follows from Theorem \ref{theorem.UPF.WR} and Corollary \ref{corollary.wr.rigidity}.
\end{proof}

\begin{prop}\label{gpcorner}
Let $\Gamma= \sG\{\Gamma_v\}$ and $\Lambda=\sH\{\Lambda_w\}$ be graph products such that: 
\begin{enumerate} \item $\Gamma_v$ and $\Lambda_w$ are icc property (T) groups for all $v\in \sV$, $w \in \sW$;
\item There is a class $\mathcal C$ of countable groups which satisfies the s-unique prime factorization property (see Definition \ref{supf}) for which $\Gamma_v$ and $\Lambda_w$ belong to $\mathcal C$, for all $v\in \sV$, $w\in \sW$.
 \end{enumerate} 
 Let $0<t<1$ be a scalar and $\Theta: \L(\Gamma)^t \ra \L(\Lambda)$ be any $\ast$-isomorphism. 
 
Then $t=1$ and there is a bijection $\sigma: {\rm cliq}(\sG)\ra {\rm cliq}(\sH)$ such that for every $\sC\in {\rm cliq}(\sG)$ there is a unitary $u_\sC\in \L(\Lambda)$ such that $\Theta (\L(\Gamma_\sC)) = u_\sC \L(\Lambda_{\sigma(\sC)})u_\sC^*$.   
\end{prop}

\begin{proof} First we observe that it suffices to show that $t=1$, as the rest of the statement follows from Theorem \ref{symmetries1}. 

{Let $\sD$ be a clique in $\sG$. Since $\L(\Gamma_\sD)$ is a II$_1$ factor,} there is a projection $p\in \L(\Gamma_\sD)$ of trace $\tau(p)=t$ with $\L(\Gamma)^t=p\L(\Gamma)p$. 
As $\L(\Gamma_\sD)$ has property (T) then so is $p \L(\Gamma_\sD ) p$. Since $p\L(\Gamma_\sD) p\subset \Theta^{-1}(\L(\Lambda)):=\N$ then by Theorem \ref{proptcliques} one can find a clique $\sF \in {\rm cliq}(\mathscr H)$ such that $p\L(\Gamma_\sD) p\prec_\N \Theta^{-1}(\L(\Lambda_\sF))$. 
Now, observe that since the inclusion {$p\L(\Gamma_\D)p\subset \N$} is irreducible, we can proceed as in the proof of Theorem \ref{symmetries1} to deduce that $p\L(\Gamma_\sD) p\nprec_\N \Theta^{-1}(\L(\Lambda_{\sF\setminus \{c\}}))$ for every $c\in \sF$. Thus, using the irreducibility condition and the moreover part of Theorem \ref{proptcliques}, there is  $u\in\sU(\Theta^{-1}(\L(\Lambda)))$ satisfying \begin{equation}\label{containment8}u p \L(\Gamma_\sD)p u^*\subset \Theta^{-1}(\L(\Lambda_{\sF})).\end{equation}
Also observe that $(up\L(\Gamma_\sD)p u^*)'\cap\Theta^{-1}(\L(\Lambda_\sF))\subseteq up (\L(\Gamma_\sD))'{\cap \L(\Gamma_\sG) pu^*}= \mathbb C p$. Hence, \eqref{containment8} is an irreducible inclusion of II$_1$ factors.

Next, since $\Theta^{-1}(\L(\Lambda_\sF))$ has property (T) and $\Theta^{-1}(\L(\Lambda_\sF))\subset p \L(\Gamma)p\subset \L(\Gamma):=\M$ then by Theorem \ref{proptcliques} one can find $\sD'\in {\rm cliq}(\mathscr G)$ such that $\Theta^{-1}(\L(\Lambda_\sF))\prec_\M \L(\Gamma_{\sD'}) $. Combining this with \eqref{containment8} we further get $p \L(\Gamma_\sD)p\prec_\M \L(\Gamma_{\sD'}) $, which further implies by Lemma \ref{corollary.graph.CI17}  that $\sD \subseteq \sD'$ and since these are cliques we conclude that $\sD=\sD'$. In conclusion, the prior intertwining relation amounts to $\Theta^{-1}(\L(\Lambda_\sF))\prec_\M \L(\Gamma_{\sD})$. Since $\L(\Gamma_\sD)$ is a II$_1$ factor we further obtain $\Theta^{-1}(\L(\Lambda_\sF))\prec_\M up\L(\Gamma_{\sD})pu^*$. Since $u\in p\M p$ is a unitary this further implies that \begin{equation}\label{intertwiningrel6}\Theta^{-1}(\L(\Lambda_\sF))\prec_{p\M p} up\L(\Gamma_{\sD})p u^*.\end{equation} By irreducibilty we have  $\Theta^{-1}(\L(\Lambda_\sF))\nprec_{p\M p} \Theta^{-1}(\L(\Lambda_{\sF\setminus\{c\}})$ for all $c\in \sF$. Thus, \eqref{intertwiningrel6}, \eqref{containment8}, and Lemma \ref{controlonesidednorm3} further imply $\Theta^{-1}(\L(\Lambda_\sF))\prec_{\Theta^{-1}(\L (\Lambda_\sF))} up\L(\Gamma_{\sD})p u^*$.  Using \cite[Lemma 2.3]{CD18}, this entails that the inclusion \eqref{containment8} has finite index, and consequently, we have  
\begin{equation}\label{cd1}
\Theta^{-1}(\L(\Lambda_{\sF}))\subset \mathscr{QN}''_{p\M p} (up \L(\Gamma_\sD) pu^*).    
\end{equation}
Since $\sD$ is a clique, Theorem \ref{controlquasinormalizer1} and Lemma \ref{QN2} imply that $QN_\Gamma(\Gamma_\sD)=\Gamma_\sD$.
Using this together with Lemma \ref{QN1} and Lemma \ref{QN2}, we obtain 
$$
\mathscr{QN}''_{p\M p} (up \L(\Gamma_\sD) pu^*)=up \mathscr{QN}''_{\M } ( \L(\Gamma_\sD) )pu^*= up \L(QN_\Gamma(\Gamma_\sD))pu^*=up \L(\Gamma_\sD)pu^*, 
$$
which together with \eqref{cd1} implies that $\Theta^{-1}(\L(\Lambda_{\sF}))\subset up \L(\Gamma_\sD)pu^*$. Together with 
\eqref{containment8} it follows that $\Theta^{-1}(\L(\Lambda_\sF))= up\L(\Gamma_{\sD})pu^*$. Finally, the strong unique prime factorization property implies $p=1$ and thus $t=1$, as desired.
\end{proof}

\subsection{Proofs of the main results}

With all the previous preparations at hand we are ready to prove the first main result, namely Theorem \ref{A}.

\begin{thm}\label{symmetries2} Let $\sG$ and $\sH$ be graphs in class ${\rm CC}_1$ and let $\Gamma= \sG\{\Gamma_v\}$ and $\Lambda=\sH\{\Lambda_w\}$ be graph product groups satisfying the following conditions: 
\begin{enumerate} \item $\Gamma_v$ and $\Lambda_w$ are icc property (T) groups for all $v\in \sV$, $w \in \sW$;
\item There is a class $\mathcal C$ of countable groups which satisfies the s-unique prime factorization property (see Definition \ref{supf}) for which $\Gamma_v$ and $\Lambda_w$ belong to $\mathcal C$, for all $v\in \sV$, $w\in \sW$. 
\end{enumerate} 

Let $t>0$ and let $\Theta : \L(\Gamma)^t \ra  \L(\Lambda)$ be any $\ast$-isomorphism. Then $t=1$ and one can find an isometry $\sigma: \sG \ra \sH$, $\ast$-isomorphisms $\theta_{i-1,i}:\L(\Gamma_{\sC_{i-1,i}})\ra \L(\Gamma_{\sC_{\sigma (\sC_{i-1,i})}}) $,  $\xi_{i}:\L(\Gamma_{\sC^{\rm int}_i})\ra \L(\Gamma_{\sigma(\sC^{\rm int}_i)})$ for all $i\in\overline{1,n}$,  and a unitary $u\in \L (\Lambda)$ such that $\Theta = {\rm ad}(u)\circ \phi_{\theta,\xi} $.

\end{thm}

\begin{proof} Without loss of generality we can assume that $t\leq 1$ and from the prior theorem we have that $t=1$. Also for simplicity of the writing we will omit $\Theta$ from the formulas. Using condition 2) in conjunction with Theorem \ref{symmetries1} one can find a bijection  $\sigma: {\rm cliq}(\sG) \ra  {\rm cliq}(\sH)$ and unitaries $u_1,\dots,u_n \in \M$ such that for any $i\in\overline{1,n}$, we have
\begin{equation}\label{unitconj6}  
u_i \L(\Gamma_{\sC_i}) u_i^*= \L(\Lambda_{\sigma(\sC_i)}).
\end{equation}

Next, condition 2. implies that for any $i\in\overline{1,n}$ there exist a complete subgraph $\mathscr D_i\subset \sigma(\mathscr C_i)$ and a unitary $\tilde u_i \in \L(\Lambda_{\sigma(\sC_i)})$  such that  $\tilde u_{i} \L(\Gamma_{\sC_{i, i+1}}) \tilde u_{i}^*= \L(\Lambda_{\sD_i} )$. Note that relation \eqref{unitconj6} still holds if we replace $u_i$ by $\tilde u_i$. Therefore, for ease of notation, we can denote $\tilde u_i$ by $u_i$.
Hence,  $ \L(\Gamma_{\sC_{i, i+1}}) =u_{i}^* \L(\Lambda_{\sD_i} ) u_{i}=u_{{i+1}}  ^*\L(\Lambda_{\sD_{i+1}} )u_{{i+1}} $ and therefore  $\L(\Lambda_{\sD_i} ) u_{i}u_{{i+1}}^*= u_{i}u_{{i+1}}  ^*\L(\Lambda_{\sD_{i+1}} )$. By Theorem \ref{controlquasinormalizer1} this further implies that $\sD_i\subseteq \sD_{i+1}$ and similarly we get  $\sD_i\supseteq \sD_{i+1}$; thus, $\sD_i= \sD_{i+1} $. Furthermore, one can see that  $u_{i} \L(\Gamma_{\sC_{i, i+1}}) u_{i}^*=u_{{i+1}}\L(\Gamma_{\sC_{i, i+1}})u_{{i+1}}^*= \L(\Lambda_{\sD_i} )$  and hence $u^*_{i} u_{i+1} \in \sN _{\M}(\L(\Gamma_{\sC_{i,i+1}}))= \L(\Gamma_{\sC_{i,i+1} \sqcup {\rm lk}(\sC_{i,i+1})})= \L(\Gamma_{\sC_i\cup\sC_{i+1}})$.  Moreover, using Proposition \ref{productunitaries} we further have that $u^*_{i} u_{i+1}= a_{i,i+1}b_{i,i+1}$ where $a_{i,i+1}\in \sU(\L(\sC_{i,i+1}))$ and $b_{i,i+1}\in \sU(\L(\Gamma_{(\sC_i\cup\sC_{i+1})\setminus \sC_{i,i+1}}))$.
To this end observe that if we let $x_{i,i+1}:= u^*_i u_{i+1}$ then we have that $x_{1,2}x_{2,3}\cdots x_{n,1}=1$. Thus, using Theorem \ref{cyclerel1} for each $ i\in \overline{1,n}$ one can find $a_i \in \sU(\L(\Gamma_{\sC_{i-1,i}}))$, $b_i \in \sU(\L(\Gamma_{\sC^{\rm int}_i}))$, $c_i \in \sU(\L (\Gamma_{\sC_{i,i+1}}))$  such that \begin{equation} u^*_i u_{i+1}= x_{i,i+1}= a_i b_i c_i b^*_{i+1} a^*_{i+2}c^*_{i+1}.\end{equation} 

Using these relations recursively together with the commutation relations and performing the apropriate cancellations we see that  for every $i\in \overline{2,n}$ we have  \begin{equation}\begin{split}u_i&= u_1 (u_1^*u_2) (u_2^*u_3) \cdots (u^*_{i-2}u_{i-1})(u^*_{i-1}u_i)\\
&= u_1 (a_1b_1c_1 b^*_2 a_3^* c_2^*)(a_2b_2c_2 b^*_3 a_4^* c_3^*) \cdots (a_{i-1}b_{i-1}c_{i-1} b_i^* a^*_{i+1} c_i^*)  \\
&= ...\\
&=u_1 a_1 b_1c_1 a_2 a_i^* b_i^* a^*_{i+1} c_i^*.\end{split}\end{equation} 
Since $a_i^* b_i^* a^*_{i+1} c_i^*\in \mathscr U(\L(\Gamma_{\sC_i}))$ we can see that by replacing each $u_i$ by $u=u_1 a_1 b_1c_1 a_2$ the relations \eqref{unitconj6} still  hold. By denoting by $\mathscr F_i = \sigma (\mathscr C_i)$ for all $i$, we observe that in particular these relations imply that $u \L(\Gamma_{\sC_{i,i+1}})u^*= \L(\Lambda_{\sF_{i,i+1}})$ for all $i$. Passing to relative commutants in each clique algebra we also have that $u \L(\Gamma_{\sC^{\rm int}_i})u^*= \L(\Lambda_{\sF^{\rm int}_i})$. We now notice that the s-unique prime factorization property of the groups implies that the map $\sigma$ arises from a isometry $\sigma: \sG\to \sH$ still denoted by the same letter. Altogether these relations give the desired statement. \end{proof}

Using the W$^*$-superrigid property (T) wreath-like product groups recently discovered in \cite{CIOS21} as vertex groups in the previous result one obtains an even more precise description of the $\ast$-isomorphisms between these von Neumann algebras; hence, we provide a proof for Theorem \ref{B}.

\begin{thm}\label{symmetries3} Let $\sG, \sH$ be graphs in class ${\rm CC}_1$ and let $G= \sG\{\Gamma_v\}$, $\Lambda=\sH\{\Lambda_w\}$ be graph product groups where all vertex groups $\Gamma_v,\Lambda_w$ are property (T) wreath-like product groups as described in the second part of Corollary \ref{corollary.supf}. 

Then for any $t>0$ and  $\ast$-isomorphism  $\Theta:\L(\Gamma)^t\ra \L(\Lambda)$ we have $t=1$ and one can find a character $\eta \in {\rm Char}( \Gamma)$, a group isomorphism  $\delta \in {\rm Isom}(\Gamma,\Lambda)$, an automorphism of $\L(\Lambda)$ of the form $\phi_{a,b}$ (see the notation after equation \eqref{branchaut}) and a unitary $u\in \L(\Lambda)$ such that $\Theta= {\rm ad}(u^*)\circ \phi_{a,b}\circ \Psi_{\eta, \delta }$.    
\end{thm}

\begin{proof} From the prior result we have $t=1$. Using Theorem \ref{symmetries2} one can find a graph isomorphism $\sigma: \mathscr G \ra \mathscr H$ and a unitary $u\in \L(\Lambda)$ such that for every clique $\mathscr C_i \in {\rm cliq}(\mathscr G)$ we have that $ u \Theta(\L(\Gamma_{\mathscr C_i}))u^*= \L(\Lambda_{\sigma(\mathscr C_i)})$. In particular, these relations imply that $ u \Theta(\L(\Gamma_{\mathscr C_{i,i+1}}))u^*= \L(\Lambda_{\sigma(\mathscr C_{i, i+1})})$ and also $u \L(\Gamma_{\mathscr C^{\rm int}_i})u^*= \L(\Lambda_{\sigma (\mathscr C^{\rm int}_i)})$ for all $i\in\overline{1,n}$. Furthermore, using Corollary \ref{corollary.wr.rigidity}  one can find unitaries $a_{i,i+1}\in \Theta(\L (\Gamma_{\mathscr C_{i,i+1}})) $ and $b_i\in \Theta(\L (\Gamma_{\mathscr C_{i,i+1}}))$ such that $\mathbb T u a_{i,i+1} \Theta(\Gamma_{\sC_{i,i+1}}) a^*_{i,i+1}u^* = \mathbb T \Lambda_{\sigma(\mathscr C_{i,i+1})}$ and $\mathbb T u b_i\Theta(\Gamma_{\mathscr C^{\rm int }_i}) b_i^*u^* = \mathbb T \Lambda_{\sigma (\mathscr C^{\rm int }_i)}$. Hence, there exists an automorphism of $\L(\Lambda)$ of the form $\phi_{a,b}$ such that by letting $\tilde\Theta=\phi_{a,b}^{-1}\circ {\rm ad} (u)\circ\Theta$, we have $\mathbb T  \tilde\Theta(\Gamma_{\sC_{i,i+1}}) = \mathbb T \Lambda_{\sigma(\mathscr C_{i,i+1})}$ and $\mathbb T\tilde\Theta(\Gamma_{\mathscr C^{\rm int }_i}) = \mathbb T \Lambda_{\sigma (\mathscr C^{\rm int }_i)}$ for any $i\in\overline{1,n}$. The conclusion trivially follows.
\end{proof}

Next, we record four immediate consequences of the prior result, and hence, provide proofs to the other main results of the introduction.

\begin{cor}\label{symmetries4} Let $\sG$ be a graph in class ${\rm CC}_1$ and let $\Gamma= \sG\{\Gamma_v\}$ be the graph product groups where all vertex groups $\Gamma_v$ are property (T) wreath-like product groups as described in the second part  of Corollary \ref{corollary.supf}. 

Then for any automorphism  $\Theta\in {\rm Aut} (\L(\Gamma))$ one can find $\eta \in {\rm Char}( \Gamma)$, $\delta \in {\rm Aut}(\Gamma)$, an automorphism of $\L(\Gamma)$ of the form $\phi_{a,b}$  and a unitary $u\in \L(\Gamma)$ such that $\Theta= {\rm ad}(u)\circ \phi_{a,b}\circ \Psi_{\eta, \delta }$.
\end{cor}

\begin{cor}\label{fg} Let $\sG$ be a graph in class ${\rm CC}_1$ and let $\Gamma= \sG\{\Gamma_v\}$ be the graph product groups where all vertex groups $\Gamma_v$ are property (T) wreath-like product groups as described in the second part of  Corollary \ref{corollary.supf}. Then the fundamental group $\mathcal F(\L(\Gamma))=\{1\}$. \end{cor}

In particular, combining these results with Theorem \ref{Thm:JC} and remark after we obtain examples when the only outer automorphisms of von Neumann algebras of graph products are the only discussed in relation \eqref{branchaut}.

\begin{cor} Let $\sG\in{\rm CC}_1$ and fix ${\rm cliq}(\sG)=\{\sC_1,\ldots, \sC_n\}$ a consecutive enumeration of its cliques. Let $\Gamma= \sG\{\Gamma_v\}$ be the graph product groups where all vertex groups $\Gamma_v$ are property (T) regular  wreath-like product groups as described in the second part of Corollary \ref{corollary.supf} which in addition are pairwise non-isomorphic, have trivial abelianization and trivial outer automorphisms. Then $${\rm Out} (\L(\Gamma))\cong  \oplus^n_{i=1} \sU(\L(\Gamma_{\sC_{i-1,i}})) \oplus \sU(\L(\Gamma_{\sC^{\rm int}_{i}})).$$

\end{cor}

\begin{proof}
Let  $\Theta\in {\rm Out}(\L(\Gamma))$. By Theorem \ref{symmetries3},   one can find a character $\eta \in {\rm Char}( \Gamma)$, a group automorphism  $\delta \in {\rm Aut}(\Gamma)$ and an automorphism of $\L(\Gamma)$ of the form $\phi_{a,b}$  such that $\Theta=  \phi_{a,b}\circ \Psi_{\eta, \delta }$. Note that for any $v\in \sV$, the restriction of $\eta$ to $\Gamma_v$ is a character of $\Gamma_v$ and by assumption, we get that $\eta(g)=1$ for any $v\in \sV$ and $g\in\Gamma_v$. Next,
recall that by
Theorem \ref{outgp} we have ${\rm Aut}(\Gamma)\cong \Gamma \rtimes  ( \left(\oplus_{v\in \sV} {\rm Aut }(\Gamma_v) \right ) \rtimes {\rm Sym }(\Gamma) )$. Now, because the vertex groups are pairwise nonisomorphic, then ${\rm Sym}(\Gamma)=1$. Moreover, since all automorphisms of the vertex groups are inner it follows that $\Psi_{\eta,\delta}$ is essentially an automorphism of the form $\phi_{a',b'}$ where $a'
$ and $b'$ are collections of unitaries implemented by group elements. In conclusion, we have that $\Theta=  \phi_{c,d}$, where $c$ and $d$ are some collections of unitaries, and the formula follows. \end{proof}

\begin{cor}\label{symmetries44}
    Let $\sG, \sH$ be graphs in class ${\rm CC}_1$ and let $G= \sG\{\Gamma_v\}$, $\Lambda=\sH\{\Lambda_w\}$ be graph product groups where all vertex groups $\Gamma_v,\Lambda_w$ are property (T) wreath-like product groups as described in the second part of Corollary \ref{corollary.supf}. 

Then for any $\ast$-isomorphism  $\Theta:C_r^*(\Gamma)\ra C_r^*(\Lambda)$, one can find a character $\eta \in {\rm Char}( \Gamma)$, a group isomorphism  $\delta \in {\rm Isom}(\Gamma,\Lambda)$, an automorphism of $\L(\Lambda)$ of the form $\phi_{a,b}$  and a unitary $u\in \L(\Lambda)$ such that $\Theta= {\rm ad}(u^*)\circ \phi_{a,b}\circ \Psi_{\eta, \delta }$.   
\end{cor}

\begin{proof}
From Lemma \ref{trivial.amenable.radical} we get that $\Gamma$ has trivial amenable radical, and hence, by \cite[Theorem 1.3]{BKKO14} it follows that $C_r^*(\Gamma)$ has unique trace. This implies that any $*$-isomorphism between  $C_r^*(\Gamma)$ and $C_r^*(\Lambda)$ lifts to a $*$-isomorphism of the associated von Neumann algebras. Now the result follows from Theorem \ref{symmetries3}.
\end{proof}




\begin{thebibliography}{99}\footnotesize{
\bibitem[Ag13]{Ag13} I. Agol, \textit{The virtual Haken conjecture}, Doc. Math. \textbf{18} (2013), 1045--1087. With an appendix by I. Agol, D. Groves, J. Manning.

\bibitem[AM10]{AM10} Y. Antol\'in, A. Minasyan,  \textit{Tits alternatives for graph products}, J. Reine Angew. Math. \textbf{704} (2015), 55--83.





\bibitem[BKKO14]{BKKO14}  E. Breuillard, M. Kalantar, M. Kennedy, N. Ozawa, {\it $C^*$-simplicity and the unique trace property for discrete groups}, Publ. Math. Inst. Hautes \'Etudes Sci. {\bf126} (2017), 35--71.






\bibitem[Ca16]{Ca16}   M. Caspers, \textit{Absence of Cartan subalgebras for right-angled Hecke von Neumann
algebras}, Anal. PDE \textbf{13} (2020), no. 1, 1--28.

\bibitem[CF14]{CF14} M. Caspers, P. Fima, \textit{Graph products of operator algebras}, J. Noncommut. Geom.
\textbf{11} (2017), no. 1, 367--411.
 

 



\bibitem[CdSS17]{CdSS17} I. Chifan, R. de Santiago, W. Sucpikarnon, {\it Tensor product decompositions of II$_{1}$ factors arising from
extensions of amalgamated free product groups}, Comm. Math. Phys. {\bf 364} (2018), 1163--1194.
		

\bibitem[CD18]{CD18} I. Chifan, S. Das, {\it A remark on the ultrapower algebra of the hyperfinite factor}, Proc. Amer. Math. Soc. {\bf 146} (2018), no. 12, 5289--5294.


\bibitem[CDD23]{CDD23} I. Chifan, M. Davis, D. Drimbe, {\it  Rigidity for von Neumann algebras of  graph product groups. II. Superrigidity results}, preprint 2023, arXiv: arXiv:2304.05500.

\bibitem[CDK19]{CDK19} I. Chifan, S. Das, K. Khan, {\it Some Applications of Group Theoretic Rips Constructions to the Classification of von Neumann Algebras}, Analysis and PDE {\bf 16} (2023) no. 2, 433-476. 



\bibitem[CDHK20]{CDHK20} I. Chifan, S. Das, C. Houdayer, K. Khan, \textit{Examples of property T factors with trivial fundamental group}, to appear in Amer. J. Math, arXiv:2003.11729.

		
\bibitem[CIOS21]{CIOS21} I. Chifan, A. Ioana, D. Osin, B. Sun, {\it Wreath-like product groups and rigidity of their von
Neumann algebras}, Ann. of Math., {\bf 198} (2023), no. 3, 1261--1303.








 

 
 






\bibitem[CI17]{CI17} I. Chifan, A. Ioana, \textit{Amalgamated free product rigidity for group von Neumann algebras}, Adv. Math. \textbf{329} (2018), 819--850.



 



\bibitem[CKE21]{CK-E21} I. Chifan, S. Kunnawalkam Elayavalli,  \textit{Cartan Subalgebras in von Neumann Algebras Associated with Graph Product Groups}, to appear in Group. Geom. Dyn., arXiv:2107.04710.


		





\bibitem[Da20]{Da20} S. Das, {\it New examples of Property (T) factors with trivial fundamental group and unique prime factorization}, Preprint arXiv:2011.04487. 

\bibitem[DGO11]{DGO11} F. Dahmani, V. Guirardel, D. Osin, \textit{Hyperbolically embedded subgroups and rotating families in groups acting on hyperbolic spaces}, Mem. Amer. Math. Soc. \textbf{245} (2017), no. 1156.



\bibitem[DHI16]{DHI16} D. Drimbe, D. Hoff, and A. Ioana, {\it Prime II$_1$ factors arising from irreducible lattices in products of rank one simple Lie groups}, J. Reine Angew. Math. {\bf 757} (2019), 197--246.  

\bibitem[DK-E21]{DK-E21} C. Ding, S. Kunnawalkam Elayavalli, Proper proximality for various families of groups, to appears in Groups Geom. Dyn, arXiv:2107.02917.





		
		


		
\bibitem[FGS10]{FGS10} J. Fang, S. Gao, and R. Smith, \textit{The Relative Weak Asymptotic Homomorphism Property for Inclusions of Finite von Neumann Algebras}, Internat. J. Math. {\bf 22} (2011), no. 7, 991--1011.



\bibitem[Ge95]{Ge95} L. Ge, {\it On maximal injective subalgebras of factors}, Adv. Math. {\bf 118} (1996), no. 1, 34-70.

\bibitem[Ge03]{Ge03} L. Ge, {\it On “Problems on von Neumann algebras by R. Kadison, 1967”}, Acta Math. Sin. (Engl. Ser.) 19
(2003), no. 3, 619–624. With a previously unpublished manuscript by Kadison; International Workshop
on Operator Algebra and Operator Theory (Linfen, 2001).


\bibitem[Gr90]{Gr90}E. Green, {\it Graph Products of Groups}, PhD Thesis, The University of Leeds, 1990, \url{http://etheses.whiterose.ac.uk/236/}.

\bibitem[GM19]{GM19}  A. Genevois, A. Martin, {\it Automorphisms of graph products of groups from a geometric
perspective}, Proc. Lond. Math. Soc. (3) \textbf{119} (2019), no. 6, 1745--1779

\bibitem[HW08]{HW08} F. Haglund, D. Wise, {\it Special cube complexes}, Geom. Funct. Anal. \textbf{17} (2008), no. 5, 1551--1620.

\bibitem[HH20]{HH20} C. Horbez, J. Huang, \textit{Measure equivalence classification of transvection-free right-angled Artin groups}, J. \'Ec. polytech. Math. \textbf{9} (2022), 1021--1067


\bibitem[HH21]{HH21} C. Horbez, J. Huang, \textit{Orbit equivalence rigidity of irreducible actions of right-angled Artin groups}, Preprint 2021, arXiv:2110.04141.

\bibitem[IPP05]{IPP05} A. Ioana, J. Peterson, and S. Popa, \textit{Amalgamated free products of weakly rigid factors and calculation of their symmetry groups}. Acta Math. {\bf 200} (2008), no. 1, 85--153. 
 
 







\bibitem [Io12b]{Io12b} A. Ioana: {\it Classification and rigidity for von Neumann algebras}, European Congress of Mathematics, EMS (2013), 601-625.

\bibitem[Io17]{Io17} A. Ioana: {\it Rigidity for von Neumann algebras}, Proceedings of the International Congress of Mathematicians-Rio de Janeiro 2018. Vol. III. Invited lectures, 1639-1672, World Sci. Publ., Hackensack, NJ, 2018.

 





	

\bibitem[Ma37]{Ma37}  W. Magnus, {\it On a theorem of Marshall Hall}, Ann. of Math. (2) {\bf 40} (1939), 764--768.

\bibitem[MO13]{MO13}  A. Minasyan, D. Osin, \textit{Acylindrical hyperbolicity of groups acting on trees}, Math. Ann. $\textbf{362}$ (2015), no. 3-4, 1055--1105.

\bibitem[MvN36]{MvN36} F.J. Murray and J. von Neumann, \textit{On rings of operators}, Ann. of Math. (2) \textbf{37} (1936), no. 1, 116--229.






\bibitem[Osi07]{Osi07}
D. Osin, Peripheral fillings o\textbf{}f relatively
hyperbolic groups, \textit{Invent. Math.} {\bf 167} (2007), no. 2,
295-326.


\bibitem[OP04]{OP04} N. Ozawa and S. Popa, {\it Some prime factorization results for type II$_{1}$ factors}, Invent. Math. {\bf 156} (2004), no. 2, 223--234.

\bibitem[OP07]{OP07}  N. Ozawa, S. Popa, {\it On a class of II$_1$ factors with at most one Cartan subalgebra}, Ann. of
Math (2) {\bf172} (2010), no. 1, 713--749.

\bibitem[Po99]{Po99} S. Popa, {\it Some properties of the symmetric enveloping algebra of a factor, with applications to amenability and property (T)}, Doc. Math. {\bf 4} (1999), 665-744.

\bibitem[Po01]{Po01} S. Popa, \textit{On a class of type II$_1$ factors with Betti numbers invariants}, Ann. of Math (2) \textbf{163} (2006), 809--899.

\bibitem[Po03]{Po03} S. Popa, \textit{Strong rigidity of $\textrm{II}_1$ factors arising from malleable actions of $w$-rigid groups I}, Invent. Math. \textbf{165}  (2006), no. 2, 369--408.


\bibitem[Po06]{Po06} S. Popa, \textit{Deformation and rigidity for group actions and von Neumann algebras}, International Congress of Mathematicians. Vol. I, 445--477, Eur. Math. Soc., Z\"urich, 2007.

\bibitem[PV06]{PV06} S. Popa, S. Vaes, Strong rigidity of generalized Bernoulli actions and computations of their symmetry groups, Adv. Math. {\bf 217} (2008), no. 2, 833--872.


		
\bibitem[PV12]{PV12} S. Popa, S. Vaes, \textit{Unique Cartan decomposition for $\rm II_1$ factors arising from arbitrary actions of hyperbolic groups}, J. Reine Angew. Math. \textbf{694} (2014), 215--239.


\bibitem[Su20]{Su20} B. Sun, {\it Cohomology of group theoretic Dehn fillings I: Cohen-Lyndon type theorems}, J.
Algebra {\bf 542} (2020), 277-307.
  
\bibitem[Va07]{Va07} S. Vaes, {\it Rigidity results for Bernoulli actions and their von Neumann algebras (after Sorin Popa)}, Séminaire Bourbaki, Vol. 2005/2006. Astérisque {\bf 311} (2007), no. 961, viii, 237--294.

\bibitem[Va08]{Va08} S. Vaes, {\it Explicit computations of all finite index bimodules for a family of II$_1$ factors}, Ann. Sci. \'Ec. Norm. Sup\'er (4) {\bf41} (2008), no. 5, 743--788.

\bibitem[Va13]{Va13} S. Vaes, {\it Normalizers inside amalgamated free product von Neumann algebras}, Publ. Res. Inst. Math. Sci. {\bf 50} (2014), no. 4, 695--721.


\bibitem[Va10a]{Va10a} S. Vaes, {\it One-cohomology and the uniqueness of the group measure space decomposition of a II$_1$ factor}, Math. Ann. {\bf355} (2013), no. 2, 661--696.

\bibitem[Va10]{Va10} S. Vaes, \textit{Rigidity for von Neumann algebras and their invariants}, Proceedings of the International Congress of Mathematicians. Volume III, 1624-1650, Hindustan Book Agency, New Delhi, 2010.

\bibitem[W11]{W11} D. T. Wise, \textit{Research announcement: the structure of groups with a quasiconvex hierarchy}, Electron. Res.
Announc. Math. Sci. $\textbf{16}$ (2009), 44--55.}

\end{thebibliography}
\end{document}